\documentclass[11pt]{amsart}
\usepackage{pifont}
\usepackage{amsthm}
\usepackage{amsfonts}
\usepackage{amssymb}
\usepackage[mathscr]{euscript}
\usepackage[all]{xy}
\usepackage[dvips]{graphics}
\usepackage{amsmath}
\usepackage[dvips,final]{graphicx}
\usepackage[dvips]{geometry}
\usepackage{color}
\usepackage{epsfig}
\usepackage{latexsym}
\usepackage{pstricks}
\usepackage{multirow}
\usepackage{rotating}
\usepackage{hyperref}
\usepackage{enumerate}

\newtheorem{theorem}{Theorem}
\newtheorem{lemma}[theorem]{Lemma}
\newtheorem{proposition}[theorem]{Proposition}

\theoremstyle{definition}
\newtheorem{definition}{Definition}[section]

\theoremstyle{remark}
\newtheorem{remark}{Remark}[section]

\begin{document}
\title{Fibered Knots and Virtual Knots}
\author{Micah W. Chrisman and Vassily O. Manturov}
\begin{abstract} We introduce a new technique for studying classical knots with the methods of virtual knot theory. Let $K$ be a knot and $J$ a knot in the complement of $K$ with $\text{lk}(J,K)=0$. Suppose there is covering space $\pi_J: \Sigma \times (0,1) \to \overline{S^3\backslash V(J)}$, where $V(J)$ is a regular neighborhood of $J$ satisfying $V(J) \cap \text{im}(K)=\emptyset$ and $\Sigma$ is a connected compact orientable $2$-manifold.  Let $K'$ be a knot in $\Sigma \times (0,1)$ such that $\pi_J(K')=K$. Then $K'$ stabilizes to a virtual knot $\hat{K}$, called a virtual cover of $K$ relative to $J$. We investigate what can be said about a classical knot from its virtual covers in the case that $J$ is a fibered knot. Several examples and applications to classical knots are presented. A basic theory of virtual covers is established.
\end{abstract}
\keywords{virtual knot, fibered knot, applications of virtual knot theory, covering, parity}
\subjclass[2000]{57M25, 57M27}
\maketitle
\section{Introduction}
\subsection{Opening Remarks} \label{opening_remarks} By \emph{classical knot theory} we mean the study of knots and links in the $3$-sphere. By \emph{virtual knot theory} we mean the study knots and and links in thickened surfaces $\Sigma \times I$ modulo stabilization, where $\Sigma$ is compact orientable surface (not necessarily closed), and $I$ is the closed unit interval. The goal of the present paper is to study classical knots using the methods of virtual knot theory. To do this, we introduce the concept of a \emph{virtual cover} of a classical knot.
\newline
\newline
Suppose that $K$ is a knot and $J$ is a knot in the complement of $K$ satisfying $\text{lk}(J,K)=0$. Let $V(J)$ denote a regular neighborhood of $J$ such that $V(J) \cap K=\emptyset$. Furthermore, suppose that the complement of $J$ admits a covering space map $\pi_J: \Sigma \times (0,1) \to \overline{S^3 \backslash V(J)}$. Let $K'$ be a knot in $\Sigma \times (0,1)$ such that $\pi_J(K')=K$. The knot $K'$ stabilizes to a virtual knot $\hat{K}$, called a \emph{virtual cover of} $K$ \emph{relative to} $J$. The aim of the present paper is to learn what can be said about the classical knot $K$ from its virtual covers.
\newline
\newline 
When $J$ is a fibered knot and $\text{lk}(J,K)=0$, virtual covers of classical knots are guaranteed to exist. This is the case considered in the present paper, although the technique could be applied more generally (for example, by using virtually fibered knots \cite{walsh}). The precise definition of a fibered knot is given below. The precise definition of a virtual cover which will be used throughout the remainder of the paper is given immediately thereafter.

\begin{definition}[Fibered Knot, Fibered Triple]  A knot $J$ in $S^3$ is said to be \emph{fibered} \cite{bz,on_knots,rolfsen} if the knot complement $\overline{S^3 \backslash V(J)}$ fibers locally trivially over $S^1$.  Let $J$ be a fibered knot with given fibration $p:\overline{S^3\backslash V(J)} \to S^1$. Let $\Sigma=p^{-1}(z_0)$ for some $z_0 \in S^1$. The triple $(J,p,\Sigma)$ is called a \emph{fiber triple}.
\end{definition}

\begin{definition}[Virtual Cover] \label{defn_virt_cov} Let $K:S^1 \to S^3$ be a classical knot and $(J,p,\Sigma)$ a fiber triple such that $K$ is in $\overline{S^3\backslash V(J)}$ and $\text{lk}(J,K)=0$. There is an orientation preserving homeomorphism from the infinite cyclic cover $M_J$ of the complement of $J$ to $\Sigma \times (0,1)$. Let $\pi_J:M_J \to \overline{S^3\backslash V(J)}$ be the covering space map. Let $K':S^1 \to M_J$ be a knot in $M_J$ satisfying $\pi_J \circ K'=K$. The lift $K'$ can be considered as a knot in $\Sigma \times I$ via the inclusion map $\Sigma \times (0,1) \hookrightarrow \Sigma \times I$. Let $\hat{K}$ denote the virtual knot representing the stability class of $K'$ in $\Sigma \times I$. A virtual knot $\hat{K}$ obtained in this way is called a \emph{virtual cover of} $K$ \emph{relative to} $(J,p,\Sigma)$. 
\end{definition}

\noindent Our main focus is to construct examples of virtual covers and apply them to problems in classical knot theory. Indeed, we will give an example of a pair of figure eight knots $K_1$, $K_2$ in $S^3$ and a trefoil $T$ in $\overline{S^3\backslash (V(K_1)\cup V(K_2))}$ such that there is no ambient isotopy taking $K_1$ to $K_2$ fixing $T$. Similarly, we will give an example of an invertible knot which cannot be transformed to its inverse without ``moving'' a fibered knot in its complement. Another example is that of an unknot $K$ in $S^3$ having a non-trivial subdiagram $D$ which is reproduced in knots $K_0$ equivalent to $K$ in $\overline{S^3 \backslash V(J)}$, where $J$ is a fibered knot. The subdiagram is reproduced in the sense that there is a smoothing of a subset of crossings of $K_0$ which results in four valent graph that is isomorphic to $D$. 
\newline
\newline
\noindent Virtual covers thus provide a new way to study classical knots with virtual knot theory. It is distinct from the usual way in which classical knots are studied with virtual knots. Typically, classical knots are considered as a \emph{subset} of the set of virtual knots. The alternative approach advocated in the present paper allows us to exploit both the non-trivial ambient topology and the intrinsic combinatorial properties of virtual knots. Indeed, both the figure eight and unknot examples described above are established by applying \emph{parity} arguments to virtual covers. Any parity for classical knots is trivial \cite{IMN}, but we see that parity arguments for virtual covers of classical knots prove to be fruitful. It is also important to note that the technique introduced in this paper is distinct from the recent work of Carter-Silver-Williams \cite{cart_sil_will}, where universal covers of surfaces are used to construct invariants of knots in thickened surfaces and virtual knots.
\newline
\newline
\noindent In addition to the examples, we give a brief theory of virtual covers. The theory will be applied to interpreting the examples. We prove that when knots are given in a special form (called \emph{special Seifert form} below), virtual covers are essentially unique. Next we investigate the relationship between virtual covers of equivalent classical knots.  If the link $J \sqcup K$ is unlinked, we show every virtual cover of $K$ is classical. It is also proved that when two equivalent knots $K_1$, $K_2$ are given in special Seifert form relative to the same fibered triple $(J,p,\Sigma)$ and the ambient isotopy taking one to the other is the identity on $V(J)$, then their virtual covers are equivalent virtual knots.  Lastly, we prove that every virtual knot is a virtual cover of some classical knot relative to some fibered triple $(J,p,\Sigma)$.
\newline
\newline
\noindent The outline of the present paper is as follows. A brief review of the four interpretations of virtual knots is given in Section \ref{sec_review}. Section \ref{virt_cov_knots} provides the technical details behind a brief theory of virtual covers. In Section \ref{sec_spec_seif}, we define special Seifert forms. The aim of Section \ref{sec_seif_form_cov} is to show that special Seifert forms have unique virtual covers relative to a given fibered triple. Section \ref{sec_prin_invar} explores the relationship between virtual covers of equivalent classical knots.  Section \ref{sec_apps} applies this theory to the three examples discussed above. Lastly, it is proved in Section \ref{sec_every_virt} that every virtual knot is a virtual cover of some classical knot relative to some fibered triple.

\subsection{Brief Review of Virtual Knot Theory} \label{sec_review} We will need four models of virtual knots: virtual knot diagrams in $\mathbb{R}^2$, knots in thickened oriented surfaces, knot diagrams on oriented surfaces (or equivalently, \emph{abstract knots} \cite{kamkam}), and Gauss diagrams. After all of the models have been described, we briefly review how one can translate one model into another. 
\newline
\newline
We begin with the virtual knot diagram interpretation. A \emph{virtual knot diagram} \cite{KaV,GPV} is an immersion $K:S^1 \to \mathbb{R}^2$ such that each double point is marked as either a \emph{classical crossing} (see top left of Figure \ref{band_pass_fig}) or a \emph{virtual crossing} (see top right of Figure \ref{band_pass_fig}). A classical crossing is the typical overcrossing/undercrossing that we have from the knot theory of embeddings $S^1 \to \mathbb{R}^3$. A virtual crossing is denoted with a small circle in the image around the double point. Two virtual knot diagrams are said to be \emph{equivalent} if they may be obtained from one another by a finite sequence of planar isotopies and the \emph{extended Reidemeister moves} (see Figure \ref{exrmoves}). Each move in the figure depicts a small ball $B \approx D^2$ (where $\approx$ means ``is homeomorphic to'') in $\mathbb{R}^2$ in which the virtual knot diagram is changed. Outside of $B$, the move coincides with the identity function $\mathbb{R}^2\backslash B \to \mathbb{R}^2 \backslash B$. 

\begin{figure}[h]
\[
\begin{array}{|c|c|c|} \hline 
\underline{\text{Reidemeister 1:}} & \underline{\text{Reidemeister 2:}} & \underline{\text{Reidemeister 3:}} \\
 \begin{array}{ccc} 
\begin{array}{c} \scalebox{.14}{\psfig{figure=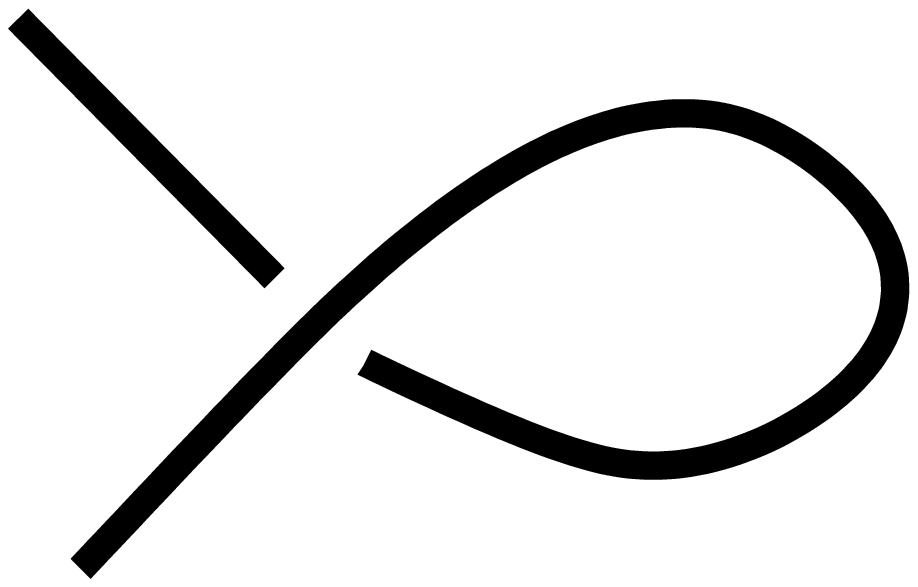}} \end{array} & \leftrightharpoons & \begin{array}{c} \scalebox{.14}{\psfig{figure=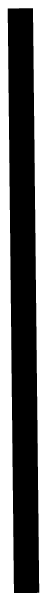}} \end{array} \\ 
 \end{array}
 &
 \begin{array}{ccc} 
\begin{array}{c} \scalebox{.14}{\psfig{figure=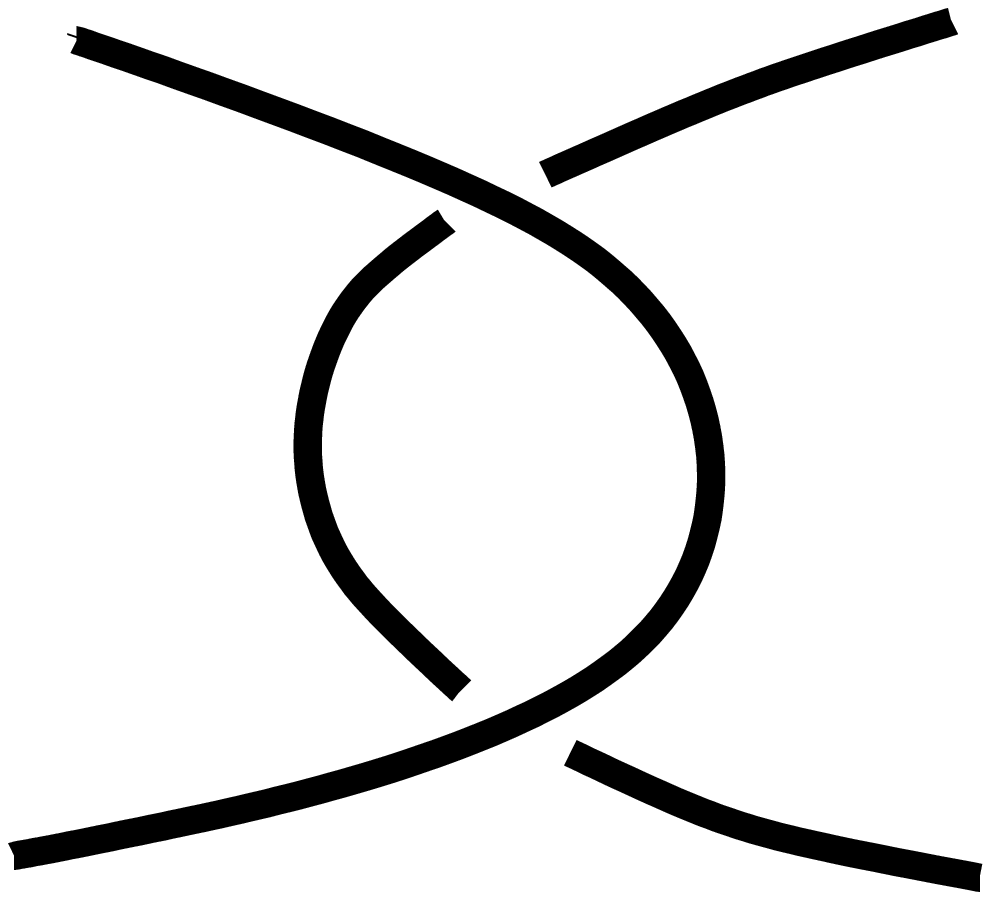}} \end{array} & \leftrightharpoons & \begin{array}{c} \scalebox{.14}{\psfig{figure=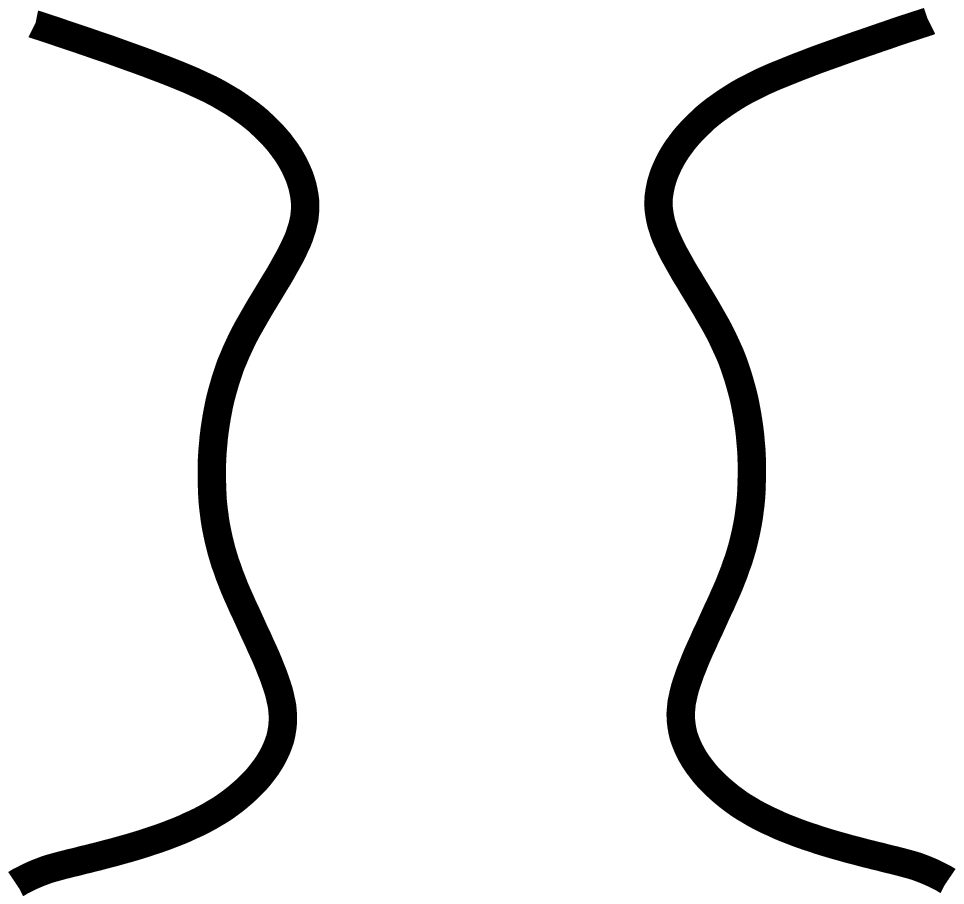}} \end{array} \\  
 \end{array}
 &
  \begin{array}{ccc}
\begin{array}{c} \scalebox{.14}{\psfig{figure=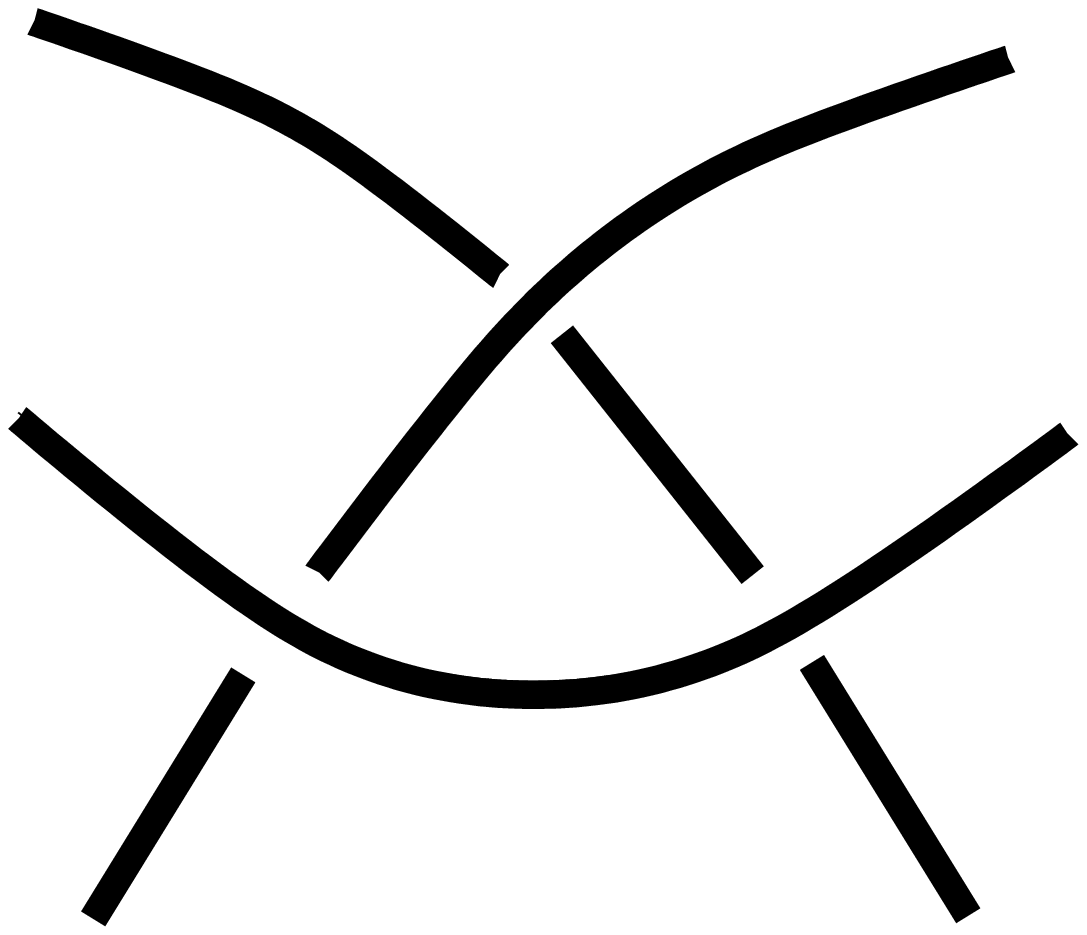}} \end{array} & \leftrightharpoons & \begin{array}{c} \scalebox{.14}{\psfig{figure=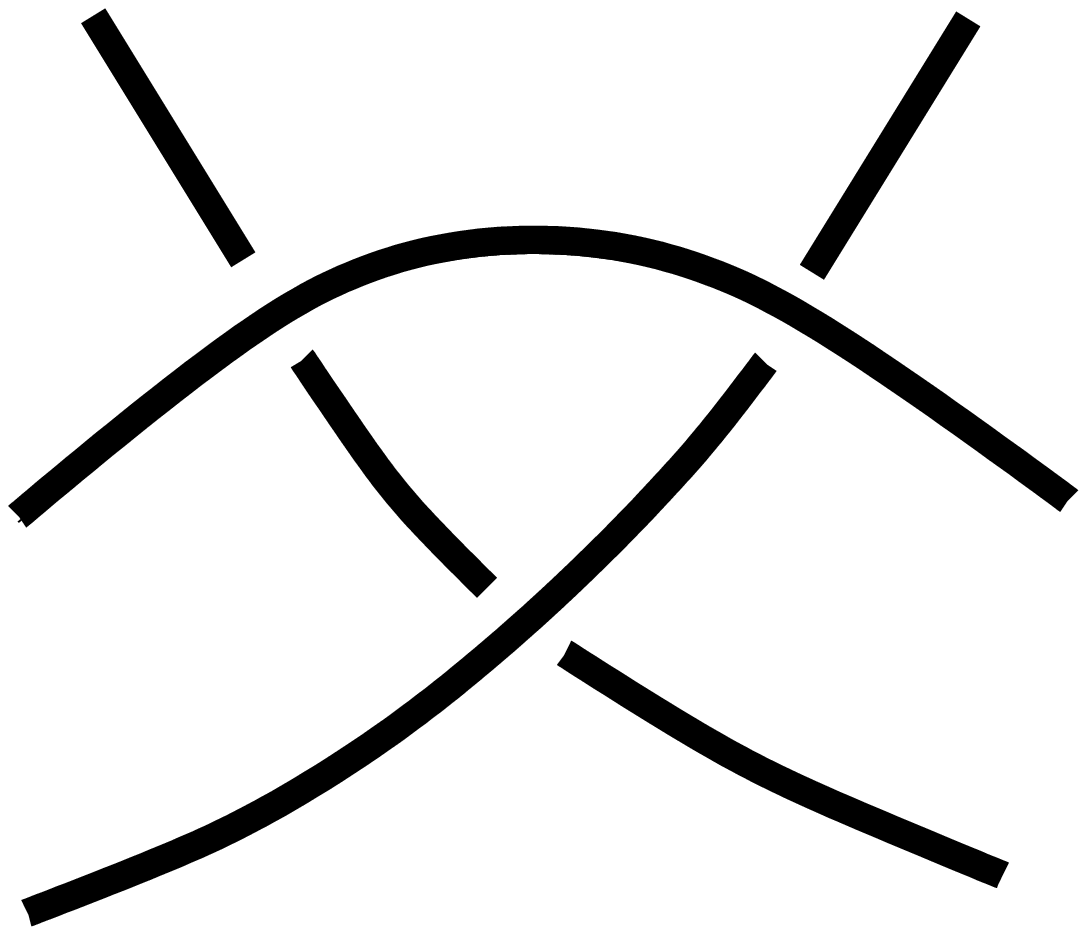}} \end{array} \\ 
 \end{array} \\ 
\hspace{1cm} & & \\ \hline 
\underline{\text{Virtual 1:}} & \underline{\text{Virtual 2:}} & \underline{\text{Virtual 3:}} \\
 \begin{array}{ccc} 
\begin{array}{c} \scalebox{.14}{\psfig{figure=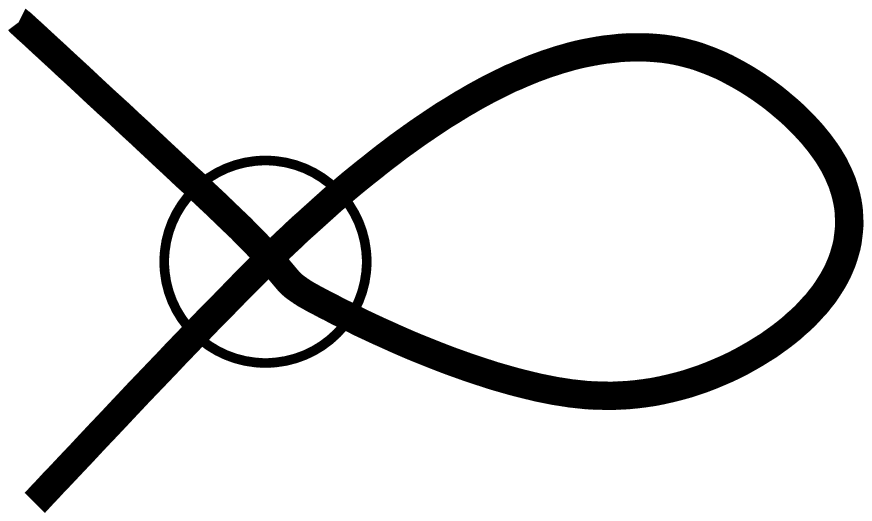}} \end{array} & \leftrightharpoons & \begin{array}{c} \scalebox{.14}{\psfig{figure=omega1_3.eps}} \end{array} \\ 
 \end{array}
 &
 \begin{array}{ccc} 
\begin{array}{c} \scalebox{.14}{\psfig{figure=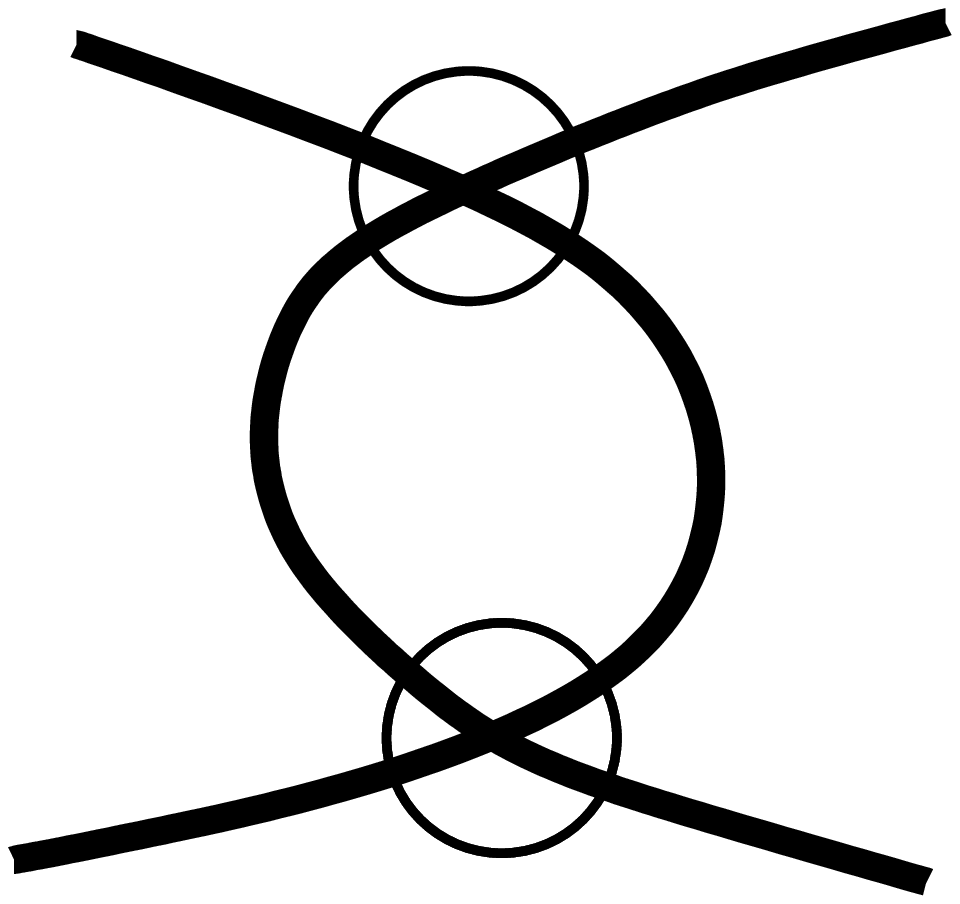}} \end{array} & \leftrightharpoons & \begin{array}{c} \scalebox{.14}{\psfig{figure=omega2_2.eps}} \end{array} \\  
 \end{array}
 &
  \begin{array}{ccc}
\begin{array}{c} \scalebox{.14}{\psfig{figure=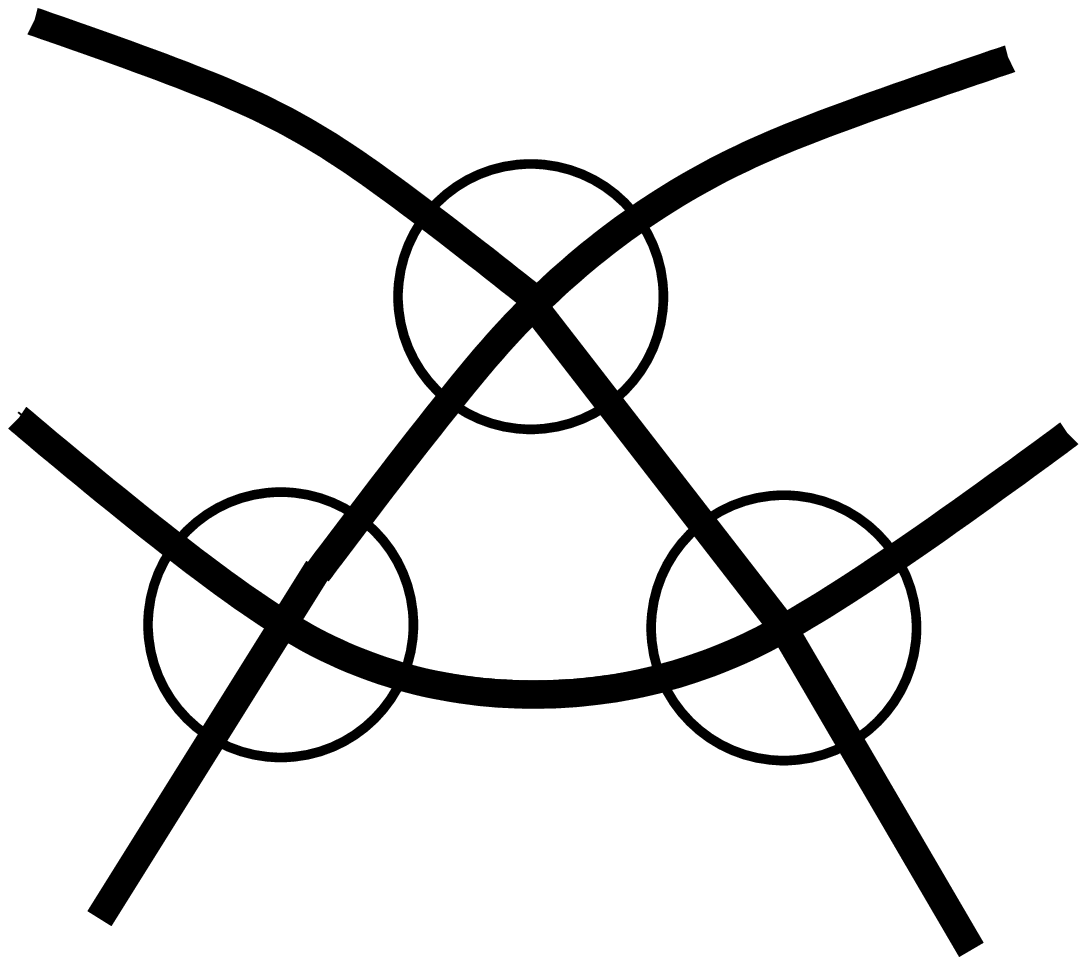}} \end{array} & \leftrightharpoons & \begin{array}{c} \scalebox{.14}{\psfig{figure=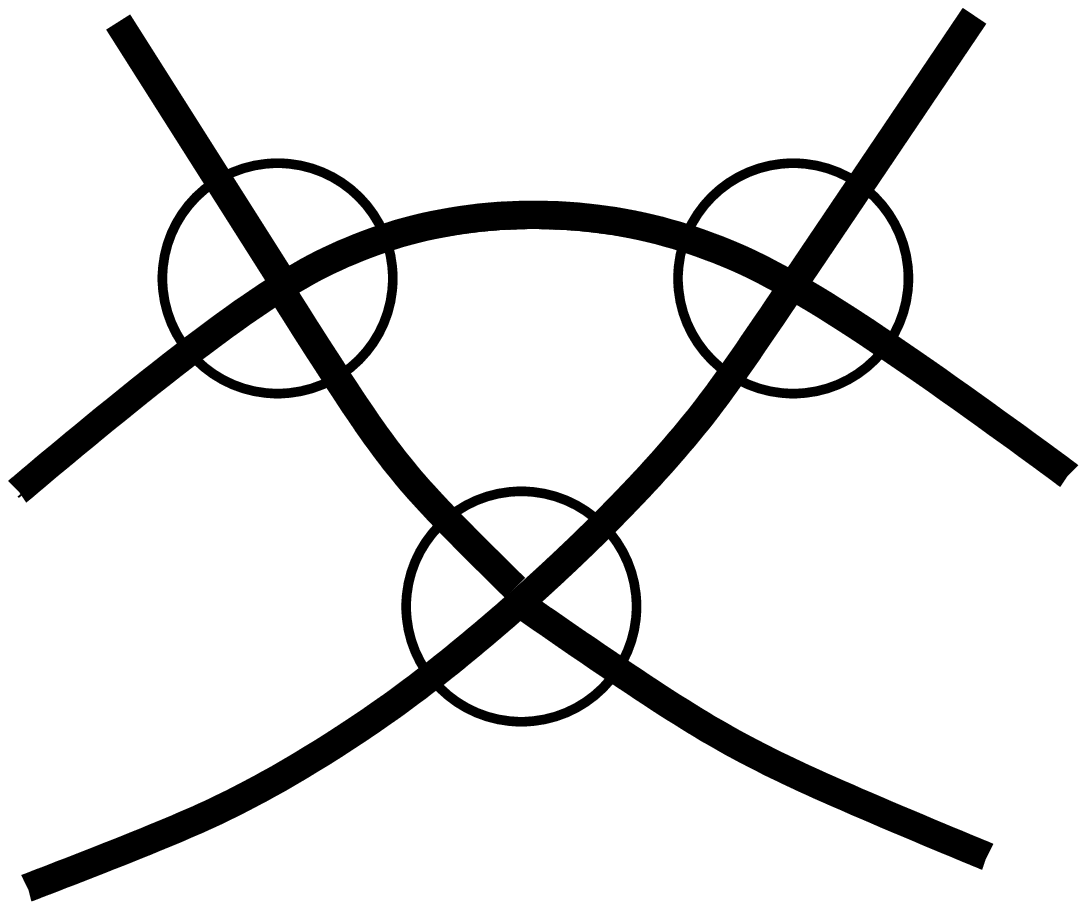}} \end{array} \\ 
 \end{array} \\ 
\hspace{1cm} & & \\ \hline
\multicolumn{3}{|c|}{\underline{\text{Virtual 4:}}} \\
\multicolumn{3}{|c|}{\begin{array}{ccc} 
\begin{array}{c} \scalebox{.14}{\psfig{figure=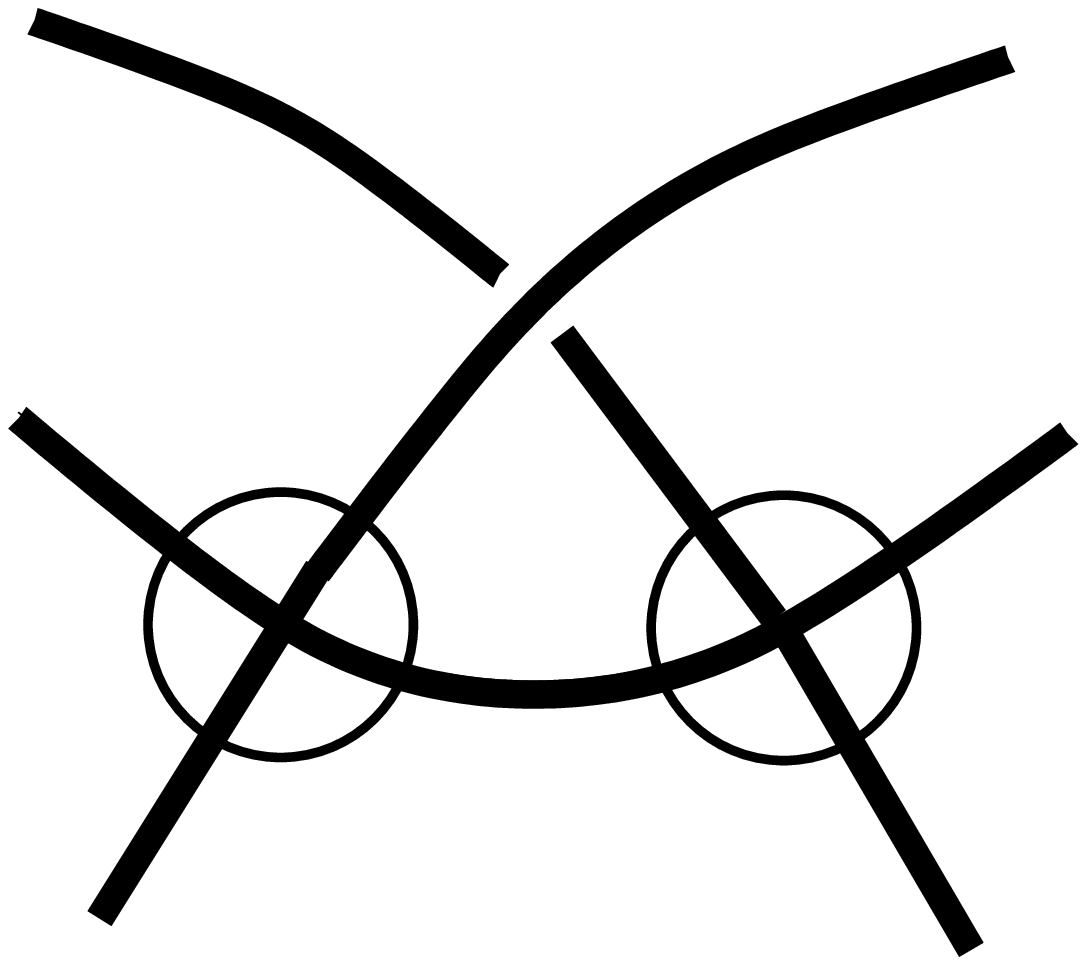}} \end{array} & \leftrightharpoons & \begin{array}{c} \scalebox{.14}{\psfig{figure=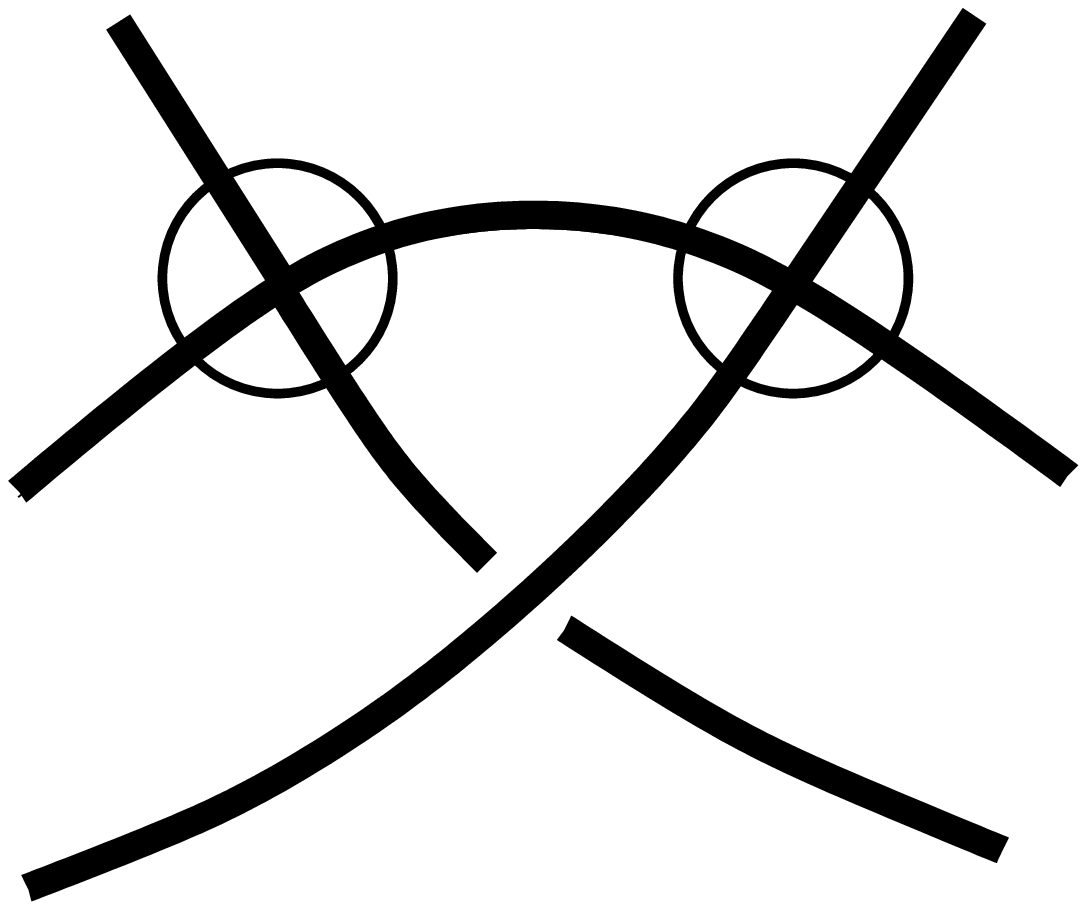}} \end{array} \\  
 \end{array}} \\ \hline
\end{array}
\]
\caption{The Extended Reidemeister Moves.} \label{exrmoves}
\end{figure}
\hspace{1cm}
\newline
The second interpretation of virtual knots is that they are \emph{knots in thickened surfaces modulo stabilization and destabilization}. Let $\Sigma$ be a compact oriented surface which is not necessarily closed. A \emph{knot in} $\Sigma \times I$ is a smooth embedding $K: S^1 \to \text{int}(\Sigma \times I)$. Two knots $K_1$, $K_2$ in $\Sigma \times I$ are said to be \emph{equivalent} if there is a smooth ambient isotopy $H:(\Sigma \times I) \times I \to \Sigma \times I$ mapping $K_1$ to $K_2$ which satisfies the property that $H_t|_{\partial(\Sigma \times I)}=\text{id}_{\partial(\Sigma \times I)}$ for all $t \in I$.  
\newline
\newline
Let $\sigma$ be a smooth embedded one-dimensional sub-manifold of $\Sigma$. A \emph{stabilization} of a knot $K$ in $\Sigma \times I$ is cutting $\Sigma \times I$ along a $\sigma \times I$ which has the property that $(\sigma \times I) \cap \text{im}(K)=\emptyset$. If $\sigma$ is homeomorphic to $S^1$, we subsequently attach a thickened disk $D^2 \times I$ along each parallel copy of $\sigma \times I$ by identifying $\sigma \times I$ with $(\partial D^2) \times I$. In addition, any connected components produced by cutting in a stabilization which do not contain $\text{im}(K)$ are discarded. A \emph{destabilization} is the inverse operation of a stabilization. The result of a stabilization is a new knot $K_1$ in the thickened surface $\Sigma_1\times I$, where $\Sigma_1$ is homeomorphic to the surface obtained from cutting $\Sigma$ along $\sigma$ and possibly deleting some components.
\newline
\newline
A knot $K_1$ in $\Sigma_1 \times I$ and a knot $K_2$ in $\Sigma_2\times I$ are said to be \emph{stably equivalent} if there is a finite sequence of equivalencies of knots in thickened surfaces, orientation preserving homeomorphisms of surfaces, and stabilizations/destabilizations which take $K_1$ to $K_2$. Let $\mathscr{TS}$ denote the set of stable equivalence classes of knots in thickened surfaces. It was proved in \cite{kamkam,CKS,kuperberg} that that there is a one-to-one correspondence $\kappa:\mathscr{TS} \to \mathscr{VK}$ between stability classes of knots in thickened surfaces and virtual knots.
\newline
\newline
The third interpretation of virtual knots is in terms of \emph{abstract knots} \cite{kamkam,CKS}. An \emph{abstract knot diagram} is a knot diagram on a compact oriented surface $\Sigma$, where $\Sigma$ is not necessarily closed.  Abstract knots on $\Sigma$ are considered up to \emph{Reidemeister equivalence} on $\Sigma$, i.e. by a sequence of Reidemeister 1, 2, and 3 moves as in the top of Figure \ref{exrmoves}. An abstract knot $\tau_1$ on $\Sigma_1$ and an abstract knot $\tau_2$ on $\Sigma_2$ are said to be \emph{elementary equivalent} if there is a compact oriented surface $\Sigma_3$ and orientation preserving embeddings $i_1:\Sigma_1 \to \Sigma_2$ and $i_2:\Sigma_2 \to \Sigma_3$ such that $i_1 \circ \tau_1$ and $i_2 \circ \tau_2$ are Reidemeister equivalent as diagrams on $\Sigma_3$. An abstract knot $\tau_1$ on $\Sigma_1$ and an abstract knot $\tau_2$ on $\Sigma_2$ are said to be \emph{stably equivalent} if there is finite sequence of elementary equivalences taking $\tau_1$ on $\Sigma_1$ to $\tau_2$ on $\Sigma_2$. 
\newline
\newline
The last interpretation of virtual knots is in terms of \emph{Gauss diagrams}. Let $K:S^1 \to \mathbb{R}^2$ be an oriented virtual knot diagram. A classical crossing $X_i$ of $K$ is a pair of points $x_{i_1},x_{i_2} \in S^1$ such that $K(x_{i_1})=X_{i}=K(x_{i_2})$. Connect the points $x_{i_1}$ and $x_{i_2}$ by a chord of $S^1$ in $D^2$. The image of a small arc in $S^1$ about $x_{i_j}$ goes to either the undercrossing or overcrossing arc of $K$ in $\mathbb{R}^2$. The chord between $x_{i_1}$ and $x_{i_2}$ is made into an arrow by directing the chord from the overcrossing arc to the undercrossing arc. Finally, we mark the sign of each classical crossing near one of $x_{i_1}$, $x_{i_2}$ with a symbol: $\oplus$ for positive crossings or $\ominus$ for negative crossings. The diagram just created is called a \emph{Gauss diagram} of $K$. Gauss diagrams are considered \emph{equivalent} up to orientation preserving homeomorphisms of $S^1$ which preserve the direction and sign of the arrows. Two Gauss diagrams are said to be \emph{Reidemeister equivalent} if they may be obtained from one another by a sequence of Gauss diagram analogs of the Reidemeister 1, 2, and 3 moves (see Figure \ref{exrmoves} and \cite{Ost,GPV}). Note that one can also find a Gauss diagram of an oriented knot diagram on a surface using the same procedure.
\newline
\newline
There is a one-to-one correspondence between any of the four interpretations \cite{kuperberg, CKS, kamkam}. The key idea in constructing the one-to-one correspondence is the \emph{band-pass presentation}. For simplicity, we describe the construction in the piecewise linear category. Let $K$ be a virtual knot diagram. A disk is drawn in the plane in a neighborhood of each classical crossing (called a \emph{cross}). Each virtual crossing corresponds to a pair of non-intersecting bands in $\mathbb{R}^3$. The bands and crosses are connected by regular neighborhoods of the regular points of $K$ in $\mathbb{R}^2$. The resulting oriented compact surface $\Sigma_K$ embedded in $\mathbb{R}^3$ is the band-pass presentation of $K$. The diagram $K'$ on $\Sigma_K$ is obtained by drawing the crossing on each ``cross'', the arcs on each ``pass'' and the regular points of $K$ on each of the regular neighborhoods (see Figure \ref{band_pass_fig}). Conversely, if you are given an oriented knot diagram $\tau$ on a surface, a corresponding virtual knot $\kappa(\tau)$ can be found by simply finding its Gauss diagram and taking the corresponding oriented virtual knot.

\begin{figure}
\[
\begin{array}{|c|c|} \hline
\underline{\text{Classical crossings get a ``cross'':}} & \underline{\text{Virtual crossing arcs go to passing bands:}} \\
\begin{array}{ccc}
\begin{array}{c}
\scalebox{.6}{\psfig{figure=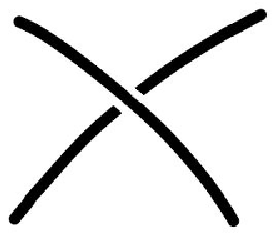}}
\end{array}
& \to \to 
&
\begin{array}{c}
\scalebox{.6}{\psfig{figure=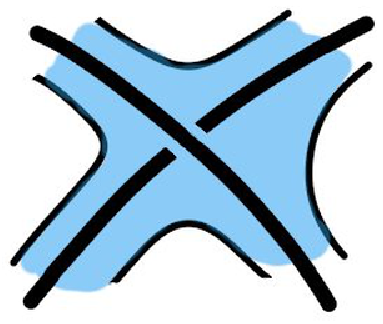}}
\end{array}
\end{array}
&
\begin{array}{ccc}
\begin{array}{c}
\scalebox{.6}{\psfig{figure=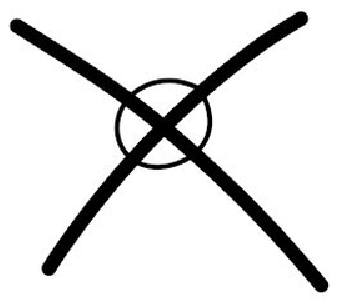}}
\end{array}
& \to \to 
&
\begin{array}{c}
\scalebox{.6}{\psfig{figure=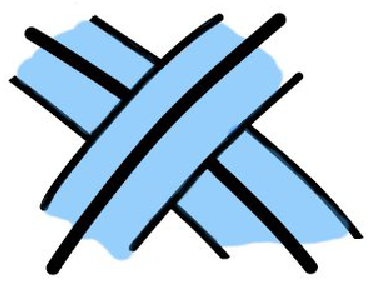}}
\end{array} \\ 
\end{array} \\ \hline
\multicolumn{2}{|c|}{\underline{\text{Crosses and passing bands are connected by regular nbhds of regular points:}}} \\ 
\multicolumn{2}{|c|}{\begin{array}{c}\scalebox{.75}{\psfig{figure=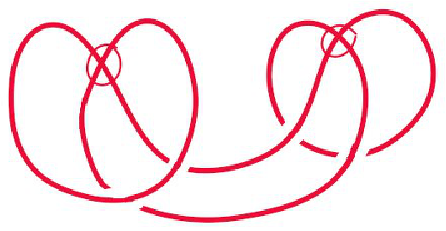}}\end{array} \to \to \to \to \begin{array}{c}\scalebox{.75}{\psfig{figure=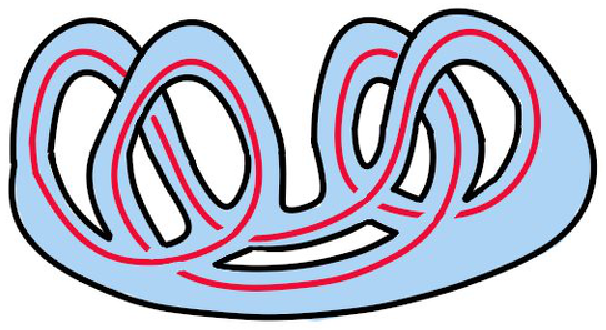}}\end{array} } \\ \hline
\end{array}
\]
\caption{Constructing the band-pass presentation of a virtual knot.} \label{band_pass_fig}
\end{figure}

\section{Theory of Virtual Coverings of Knots}\label{virt_cov_knots}
\subsection{Special Seifert Form} \label{sec_spec_seif} Virtual coverings of a knot $K$ relative to a fibered triple $(J,p,\Sigma)$ can be easily determined when the link $J \sqcup K$ in $S^3$ is presented in \emph{special Seifert form}. A special Seifert form consists, roughly, of a Seifert surface $\Sigma$ of $J$ such that the image of $K$ is contained in $\Sigma$ except in finitely many $3$-balls. To define this more precisely, we will begin with the definition of a \emph{Seifert surface of a knot}.

\begin{definition}[Seifert surface] Let $M$ be a $3$-manifold and $J$ a knot in $M$. A \emph{Seifert surface of J} is an embedded (p.l. or smooth) compact orientable $2$-manifold $\Sigma$ in $M$ such that $\text{im}(J)=\partial \Sigma$.  
\end{definition}

\begin{remark} For a fibration of a fibered knot $J$, every fiber $\Sigma$ may be considered a Seifert surface of $J$ by identifying $J$ with $V(J) \cap \Sigma$. Each fiber is of minimal genus \cite{bz}. The Seifert surface obtained in this way is unique up to isotopy in $S^3$ \cite{whitten}.
\end{remark}

\begin{figure}[htb]
\[
\begin{array}{ccc}
\begin{array}{c}
\scalebox{.75}{\psfig{figure=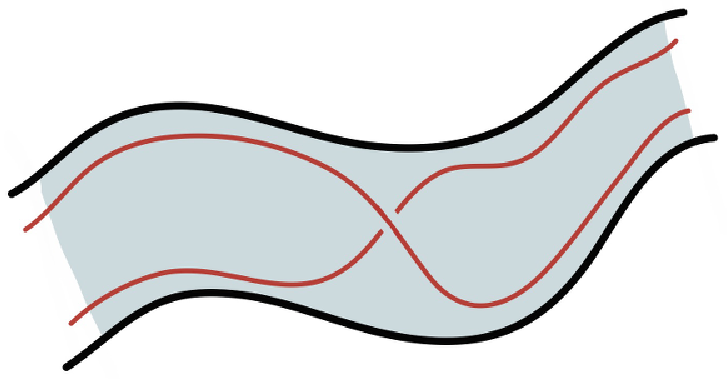}} \end{array} &\to\to &
\begin{array}{c}
\scalebox{1}{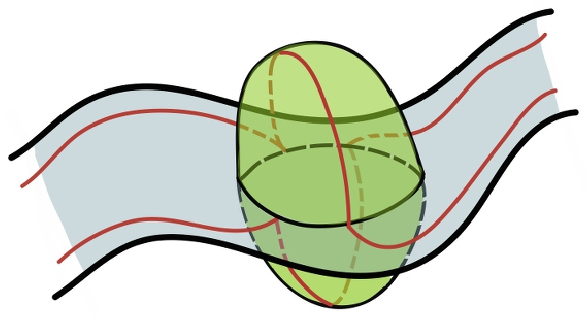}\end{array}
\end{array}
\]
\caption{Part of a knot diagram on a surface (left) and its corresponding special Seifert form (right).}\label{prop_band_fig}
\end{figure}

\begin{definition}[Special Seifert Form $(K;J,\Sigma)$] Let $L=J \sqcup K$ be a two component link in a orientable, compact, connected, (p.l. or smooth) $3$-manifold $M$. Let $\Sigma$ be a Seifert surface of $J$. Suppose that there are $3$-cells $B_1,\ldots,B_n$ in $M$ such that each $B_i$ is in a coordinate neighborhood of some point $z_i \in \text{int}(\Sigma)$ and such that the following properties are satisfied for all $i$, $1 \le i \le n$.
\begin{enumerate}
\item $B_i \cap \Sigma$ is a closed disk $D_i \subset \text{int}(\Sigma)$. 
\item $B_i \cap K$ consists of two disjoint arcs $a_{i_1}$, $a_{i_2}$ in $\partial B_i$. 
\item $(a_{i_1} \cup a_{i_2}) \cap D_i$ is the set of endpoints of the arcs $a_{i_1}$ and $a_{i_2}$ and the interiors of the two arcs are contained in different connected components of $B_i \backslash D_i$ (see the right hand side of Figure \ref{prop_band_fig}).
\item $K \subset \Sigma \cup \bigcup_{i=1}^n B_i$.
\item $K \cap \Sigma$ is a union of a finite number of pairwise disjoint closed intervals.
\end{enumerate}
In this case, we will say that $L$ is in \emph{special Seifert form}. A special Seifert form of $L$ is denoted $L=(K;J,\Sigma)$.
\end{definition}

\begin{remark} \label{remark_lk_0_1} It follows directly from the definition that a special Seifert form $(K;J,\Sigma)$ has the property that $\text{lk}(J,K)=0$. For a combinatorial argument of this observation, see Remark \ref{remark_lk_0_2} below.  
\end{remark}

\noindent Given a special Seifert form, one can find a knot diagram on a surface.  To do so consistently, we define the upper and lower hemisphere of each given $3$-cell in the definition of a special Seifert form.

\begin{definition}[Upper/Lower Hemisphere] Let $(K;J,\Sigma)$ be a special Seifert form for $J \sqcup K$. Suppose also that $\Sigma$ is smooth and oriented. Let $N \approx \Sigma \times (-1,1)$ be an open tubular neighborhood of $\text{int}(\Sigma)$ such that $\Sigma$ is identified with $\Sigma \times \{0\}$ and such that for each given $3$-ball $B$ from the definition of special Seifert form, we have that $N \cap B \approx D^2 \times (-1,1)$. We assume that the homeomorphism between $N$ and $\Sigma \times (-1,1)$ is orientation preserving, where $(-1,1)$ is given the standard orientation. Define $N^+$($N^{-1}$) to be the component of $N\backslash \Sigma$ corresponding to $\Sigma \times (0,1)$ (resp. $\Sigma \times (-1,0)$). If $B_i$ is a $3$-ball from the definition of $(K;J,\Sigma)$, the \emph{upper (lower) hemisphere of} $B_i$ \emph{relative to} $N$ is the component of $B_i \backslash (B_i \cap \Sigma)$ which intersects $N^+$(resp. $N^{-}$). 
\end{definition}

\begin{definition}[Diagram of Special Seifert Form] Let $(K;J,\Sigma)$ be a special Seifert form with set of given $3$-cells $\{B_i=U_i \cup L_i \cup D_i:1 \le i \le n\}$, where each $U_i$ ($L_i$) is the upper (resp. lower) hemisphere of $B_i$ relative to some open tubular neighborhood $N$ of $\text{int}(\Sigma)$. In each $D_i=B_i \cap \Sigma$, we connect the two points of $a_{i_1} \cap D_i$ by a smooth arc $b_{i_1}$ in $D_i$ and the two points $a_{i_2} \cap D_i$ by a smooth arc $b_{i_2}$ in $D_i$. We may assume that $b_{i_1}$ and $b_{i_2}$ intersect exactly once transversally. The smooth $b_{i_j}$ which connects the endpoints of the arc in the upper hemisphere of $B_i$ is designated as the over-crossing arc and $b_{i_{3-j}}$ is designated as the under-crossing arc. We create a \emph{knot diagram} $[K;J,\Sigma]$ \emph{on} $\Sigma$ \emph{of the special Seifert form} $(K;J,\Sigma)$ by connecting the arcs $b_{i_1},b_{i_2}$ and the arcs of $K \cap \Sigma$ (and smoothing appropriately).
\end{definition}

\begin{remark} Diagrams $[K;J,\Sigma]$ on $\Sigma$ of special Seifert forms $(K;J,\Sigma)$ are well-defined in the sense that any two diagrams are equivalent as knot diagrams on $\Sigma$.
\end{remark}

\subsection{Special Seifert Forms and Virtual Covers} \label{sec_seif_form_cov} In the present section, it is proved that for a given fibered triple $(J,p,\Sigma)$ and knot $K$ having special Seifert form $(K;J,\Sigma)$, the virtual covers relative to $(J,p,\Sigma)$ are unique up to equivalence of virtual knots.  The lemma also provides a simple method by which to find this unique virtual cover.

\begin{lemma} \label{special_ok} Suppose $\hat{K}$ is a virtual cover of $K$ relative to $(J,p,\Sigma)$ and that $(K;J,\Sigma)$ is in special Seifert form in $\overline{S^3 \backslash V(J)}$. Then $\hat{K} \leftrightharpoons \kappa([K;J,\Sigma])$ as virtual knots.
\end{lemma}
\begin{proof} The covering space $M_J$ of $\overline{S^3\backslash V(J)}$ may be considered as the induced bundle of the exponential map $\text{exp}:\mathbb{R} \to S^1$, defined by $\text{exp}(t)=e^{2\pi i t}$, and the fibration $p:\overline{S^3\backslash V(J)} \to S^1$. This gives the following commutative diagram \cite{rolfsen}.
\[
\xymatrix{M_J  \ar[r]^{p'} \ar[d]_{\pi_J} & \mathbb{R} \ar[d]^{\text{exp}} \\
\overline{S^3\backslash V(J)} \ar[r]_p & S^1
}
\]
There is a $z_0 \in S^1$ such that $\Sigma=p^{-1}(z_0)$. Let $i:\Sigma \to \overline{S^3\backslash V(J)}$ denote the inclusion. Let $y_0 \in \text{im}(K) \cap \Sigma$. Let $x_0 \in \pi_J^{-1}(y_0)$.  There is a lift $i_{x_0}:(\Sigma,y_0) \to (M_J,x_0)$. Let $t_0=K^{-1}(x_0)$ be the basepoint of $S^1$. Since $\hat{K}$ is a virtual cover of $K$ relative to $(J,p,\Sigma)$, we have by Definition \ref{defn_virt_cov} that $\text{lk}(J,K)=0$. Hence $K:S^1 \to \overline{S^3\backslash V(J)}$ lifts to a simple closed curve $K_{x_0}:(S^1,t_0) \to (M_J,x_0)$ \cite{rolfsen}, mapping the the basepoint of $S^1$ to the basepoint of $M_J$. The lift $K_{x_0}$ must also be smoothly embedded, and hence we have that $K_{x_0}$ is a knot in $M_J$.
\newline
\newline
Let $\{B_j:1 \le j \le n\}$ denote the set of $3$-cells in the special Seifert form $(K;J,\Sigma)$. We may assume that each $3$-cell $B_j$ is sufficiently small that it is contained in a neighborhood $C_j$ which is evenly covered by $\pi_J$. Hence, for all $j$, $\pi_J^{-1}(B_j)$ is a disjoint union of $3$-cells in $M_J$. Let $B_{j,x_0}$ denote the unique 3-cell in this disjoint union constituting $\pi_J^{-1}(B_j)$ such that $B_{j,x_0} \cap i_{x_0}(\Sigma) \ne \emptyset$. Then $K_{x_0}(S^1) \subset i_{x_0}(\Sigma) \cup \bigcup_j \partial B_{j,x_0}$. It follows that a Gauss diagram for $[K_{x_0};\partial (i_{x_0}(\Sigma)),i_{x_0}(\Sigma)]$ is given by a Gauss diagram for $[K;J,\Sigma]$. Thus, the virtual cover corresponding to $K_{x_0}$ is equivalent as a virtual knot to $\kappa([K;J,\Sigma])$ (see Figure \ref{special_ok_fig}).
\end{proof}

\begin{figure}[htb]
\[
\xymatrix{
\begin{array}{c} \scalebox{.35}{\epsfig{figure=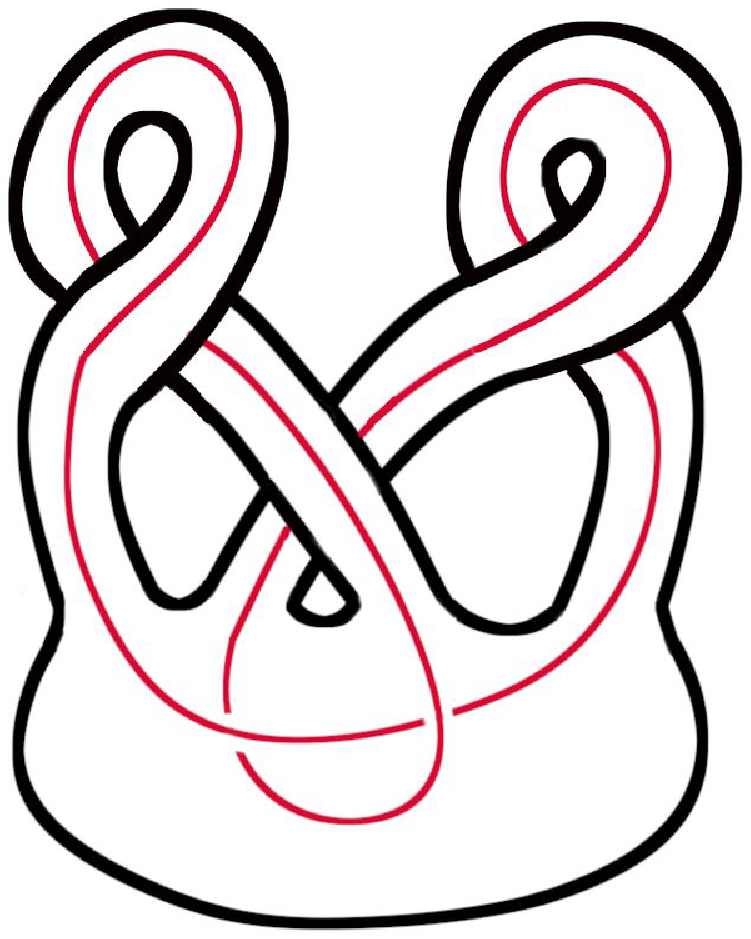}} \end{array} 
\ar[r]^{\kappa} & \begin{array}{c} \scalebox{.35}{\epsfig{figure=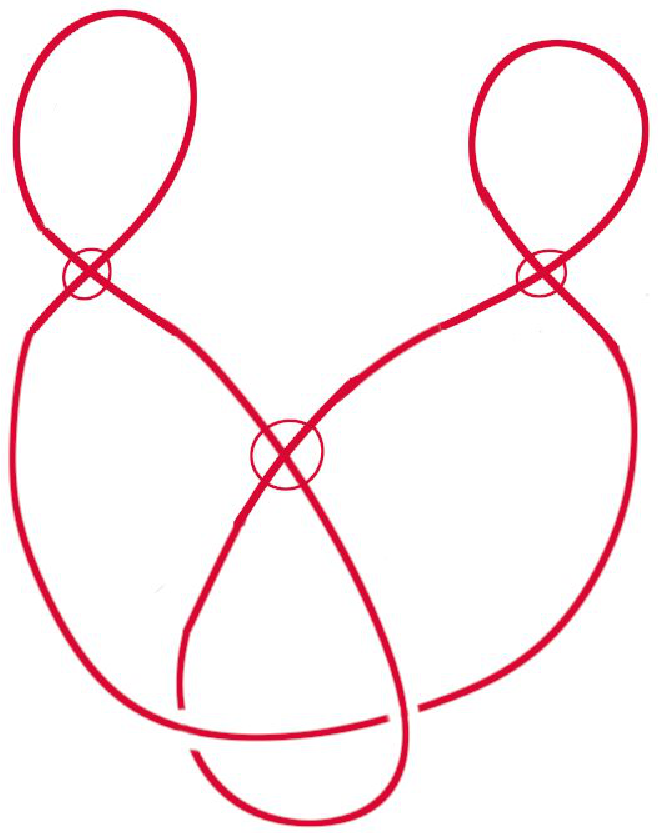}} \end{array}
}
\]
\caption{(Left) A knot diagram $[K;J,\Sigma]$ on $\Sigma$ of a special Seifert form $(K;J,\Sigma)$. The knot $J$ is a trefoil and $K$ is an unknot. (Right) A virtual cover of $K$ relative to $(J,p,\Sigma)$.}\label{special_ok_fig}
\end{figure}

\subsection{Principles of Invariance for Virtual Coverings} \label{sec_prin_invar} We consider the question of how virtual covers behave under equivalence of classical knots and links in $S^3$. We first show that virtual covers can be used to determine that the two component link $J \sqcup K$ is not unlinked, where $K$ is a classical knot, $J$ is fibered, and $\text{lk}(J,K)=0$. 

\begin{theorem} \label{unlinked_implies_classical} Let $K$ be a knot and $J$ a fibered knot in the complement of $K$ such that $\text{lk}(J,K)=0$. Let $p:\overline{S^3\backslash V(J)} \to S^1$ be a given fibration and $\Sigma=p^{-1}(z_0)$ for some $z_0 \in S^1$. If $J$ and $K$ are unlinked in $S^3$, then every virtual cover $\hat{K}$ of $K$ relative to $(J,p,\Sigma)$ is classical.
\end{theorem}
\begin{proof} Since $J$ and $K$ are unlinked, there is a $3$-cell $V$ in $S^3$ such that $\text{im}(K) \subset V$ and $\text{im}(J) \cap V=\emptyset$. By applying a contraction in $V$, we may assume that there is a 3-cell $V' \subset V$ such that $\text{im}(K) \subset V'$ and there is a neighborhood $U$ of $V'$ which is evenly covered by $\pi_J$.
\newline
\newline
If $\pi_J \circ K'=K$, it follows that $K'$ is contained in a $3$-cell in $M_J$.  Thus, there must be a sequence of stabilizations of $\Sigma \times I \supset M_J$ to $D^2 \times I$. Thus, $K'$ stabilizes to a classical knot.
\end{proof}
\hspace{1cm}
\newline
\noindent We next consider the question of ambient isotopies of knots $K$. If $K_1 \leftrightharpoons K_2$ as classical knots in $S^3$, what is the relationship between a virtual cover of $K_1$ and a virtual cover of $K_2$, relative to some fibered triple $(J,p,\Sigma)$? The following lemma gives a sufficient condition under which an ambient isotopy in the complement of the fibered component $J$ lifts to an ambient isotopy in the infinite cyclic cover of the complement of $J$. After the lemma is proved, we apply it to the case of virtual covers.

\begin{lemma} \label{isotopies_lift} Let $K:S^1 \to S^3$ be a classical knot and $(J,p,\Sigma)$ a fibered triple such that $K$ is in $\overline{S^3\backslash V(J)}$ and $\text{lk}(J,K)=0$. Let $H:S^3 \times I \to S^3$ be a smooth ambient isotopy between $K$ and $H_1(K)$ such that $H_t|_{V(J)}=\text{id}_{V(J)}$ for all $t \in I$. Let $K':S^1 \to M_J$ be a knot in $M_J$ satisfying $\pi_J \circ K'=K$. Then there is a smooth ambient isotopy $H':M_J \times I \to M_J$ such that $\pi_J \circ H_1'(K')=H_1(K)$, and $H_t'|_{\partial M_J}=\text{id}|_{\partial M_J}$ for all $t \in I$.  
\end{lemma}
\begin{proof} Let $t_0 \in S^1$ and $y_0=K(t_0)$. Then we may write $K$ as a map of pointed spaces $K:(S^1,t_0) \to (S^3,y_0)$. Since $\text{lk}(J,K)=0$, any lift of $K$ to $M_J$ must be a simple closed curve. Since the lift of $K$ to $M_J$ can be taken to be a smooth embedding, it follows that any lift of $K$ is a knot in $M_J$. There is a $x_0 \in \pi_J^{-1}(y_0)$ such that the lift $K_{x_0}:(S^1,t_0) \to (M_J,x_0)$ of $K$ is the given lift $K'$.
\newline
\newline
Let $F:S^1 \times I \to \overline{S^3\backslash V(J)}$ be the homotopy defined by $F(z,t)=H(K(z),t)$. Then $F$ lifts to a homotopy $F':S^1 \times I \to M_J$ satisfying $F'(z,0)=K'(z)$. Now, for all $t \in I$, $H_t(K)$ is a knot in $\overline{S^3\backslash V(J)}$. Thus the lifts of $H_t(K)$ must also be smoothly embedded. It follows that $F_t':S^1 \to M_J$ must also be a knot in $M_J$. Thus, $F'$ is an isotopy.
\newline
\newline
By inclusion of $M_J$ into $\Sigma \times I$, we may consider $F'$ as an isotopy $S^1 \times I \to \Sigma \times I$ whose image does not intersect $\partial (\Sigma \times I)$. Hence, by the isotopy extension theorem (see \cite{hirsch}, Chapter 8, Theorem 1.3) there is a smooth ambient isotopy $H': (\Sigma \times I ) \times I \to \Sigma \times I$ taking $K'$ to $F_1'(K')=H_1'(K')$ and satisfying $H_t'|_{\partial (\Sigma \times I)}=\text{id}|_{\partial (\Sigma \times I)}$. Then $\pi_J \circ H_1'(K')=\pi_J\circ F_1'(K')=F_1(K)=H_1(K)$.  
\end{proof}

\begin{remark} \label{loc_unknot} A weaker version of Lemma \ref{isotopies_lift} can be proved in the piecewise linear category. In this case we must add the hypotheses that $\Sigma \not \approx D^2$ and $K_1,K_2$ are \emph{locally unknotted} in $\overline{S^3 \backslash V(J)}$. For a useful definition of locally unknotted, see \cite{nanyes}. With these additional hypotheses, the lemma follows from \cite{feustel}, Theorem 3.4 and the proof Theorem 3.3.
\end{remark}

\begin{theorem} \label{main_equiv_thm} For $i=1,2$, let $K_i:S^1 \to S^3$ be a classical knot and $(J,p,\Sigma)$ a fibered triple such that $K_i$ is in $\overline{S^3\backslash V(J)}$. For $i=1,2$, suppose that $(K_i;J,\Sigma)$ is in special Seifert form (so that $\text{lk}(J,K_i)=0$). Let $\hat{K}_1$, $\hat{K}_2$ be virtual covers of $K_1$, $K_2$, respectively, relative to $(J,p,\Sigma)$. If there is a smooth ambient isotopy $H:S^3 \times I \to S^3$ taking $K_1$ to $K_2$ such that $H_t|_{V(J)}=\text{id}_{V(J)}$ for all $t \in I$, then $\hat{K}_1 \leftrightharpoons \hat{K}_2$ as virtual knots.
\end{theorem}
\begin{proof} By Lemma \ref{isotopies_lift}, there is a smooth ambient isotopy between a lift $K_1'$ of $K_1$ to $M_J$ and a lift $K_2'$ of $K_2$ to $M_J$. Hence, if $K_1'$ and $K_2'$ are considered as knots in $\Sigma \times I$ (via inclusion), then they must stabilize to virtual knots $V_1$ and $V_2$, respectively, which are equivalent virtual knots. Since we have for $i=1,2$ that $(K_i;J,\Sigma)$ is a special Seifert form, it follows from Lemma \ref{special_ok} that $V_i \leftrightharpoons \kappa([K_i;J,\Sigma]) \leftrightharpoons \hat{K}_i$. Thus, $\hat{K}_1 \leftrightharpoons \hat{K}_2$ as virtual knots.
\end{proof}

\section{Applications and Examples} \label{sec_apps}

\noindent Classical knot theory is a part of virtual knot theory: it follows from Kuperberg's theorem that if two classical knots are equivalent as virtual knots then they are ambient isotopic. On the other hand, virtual knots considered as knots in thickened surfaces modulo stabilization/destabilization, have a rich topology of the ambient space (indeed, it is $\Sigma \times I$, where $\Sigma$ is a compact orientable surface). This non-trivial topology allows one to extend many invariants of virtual knots by introducing some extra topological/combinatorial data (see \cite{denisvass_book}). 
\newline
\newline
One of the main approaches of such sort uses the parity theory introduced by the second named author \cite{Ma2,on_free_knots}. Such invariants rely on homology and homotopy information, and in some cases they allow us to reduce questions about knots to questions about their representatives (see Theorem \ref{free_reproduced}): if a knot diagram is ``odd enough'' or ``complicated enough'' then it reproduces itself in any equivalent diagram. The invariants constructed in this way (the parity bracket etc.) contain some graphical information about the knot which appears in every representative of the knot. 
\newline
\newline
These methods cannot be applied directly to classical knot theory because of the trivial topology of the ambient space $\mathbb{R}^{3}$ and the absence of parity for knots. Nevertheless, the theory of classical knots (and, in fact, links) can be put into the framework of virtual knot theory by using methods described in the previous two sections. This allows one to apply the invariants and constructions already discovered for virtual knots to the case of classical knots. This is the aim of the present section. We begin with a brief review of parity.

\subsection{Brief Review of Parity} \label{sec_parity_review} The canonical example of a parity is the \emph{Gaussian parity}. For a Gauss diagram $D_K$ of a virtual knot $K$, two arrows are said to \emph{intersect} (or be \emph{linked}) if the chords of $S^1$ between the endpoints of the arrows intersect as lines in $D^2$. 
\newline
\newline
A classical crossing of $K$ is said to be \emph{odd} if its corresponding arrow in $D_K$ intersects an odd number of arrows in $D_K$. A classical crossing that is not odd is said to be \emph{even}. Observe that (1) the crossing in a Reidemeister one move is even, (2) the two crossings involved in a Reidemeister 2 move both have the same parity, and (3) an even number of the crossings in a Reidemeister 3 move are odd.
\newline
\newline
A \emph{parity} is, roughly, any function on crossings of $K$ that satisfies properties (1)-(3) above.  Parities have been used to extend many invariants of virtual knots \cite{Af,IMN,CM,mant_par_cobord}. For example, the \emph{parity bracket} \cite{IMN,kauff_new} uses parity to extend the Kauffman bracket. A general theory of parity based on an axiomatic approach for virtual knots, flat knots, free knots, and curves on surfaces has been developed \cite{IMN,denisvass_book}.
\newline
\newline
The most elementary use of parity to create a virtual knot invariant is the \emph{odd writhe} \cite{KaV}. Let $P(K)$ denote the number of crossings of $K$ which are odd in the Gaussian parity and signed $\oplus$. Let $N(K)$ denote the number of crossings of $K$ which are odd in the Gaussian parity and signed $\ominus$. The odd writhe is defined to be:
\[
\theta(K)=P(K)-N(K).
\]
Since every classical knot has a diagram in which all of the crossings are even, it follows that $\theta(K)=0$ for all classical knots $K$. For virtual knots, however, the invariant is useful. If $\theta(K)\ne 0$, then one can immediately conclude that $K$ is non-classical.

\subsection{Isotopies of Knots Fixing a Knot in the Complement}\label{sec_odd_writhe} In light of Theorem \ref{main_equiv_thm}, we see that virtual covers can be used to find examples of equivalent knots in $S^3$ for which every ambient isotopy taking one to the other ``moves'' a knot in the complement. This is a result about classical knot theory that is established using techniques which appear only in the theory of virtual knots. The following proposition gives a typical application of our technique.

\begin{figure}[h]
\[
\xymatrix{
\begin{array}{c} \scalebox{.4}{\psfig{figure=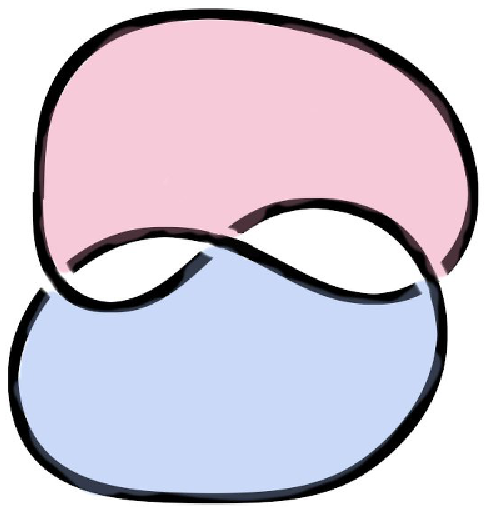}} \end{array} \ar[r]^{H} & \begin{array}{c} \scalebox{.4}{\psfig{figure=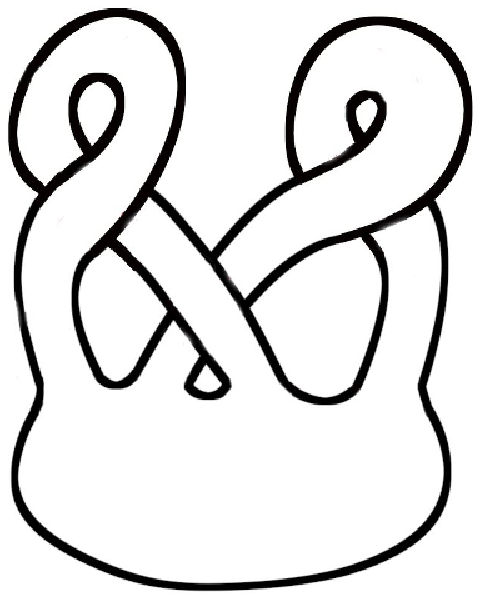}} \end{array}
}
\]
\caption{There is an ambient isotopy of $S^3$ taking the Seifert surface of the trefoil on the left to the Seifert surface of the trefoil on the right.} \label{band_trefoil}
\end{figure}

\begin{proposition}\label{prop_fig_8} There exist figure eight knots $K_1$ and $K_2$ in $S^3$ and a trefoil $T$ in the mutual complement $\overline{S^3 \backslash (V(K_1) \cup V(K_2))}$ such that there is no ambient isotopy $H:S^3 \times I \to S^3$ taking $K_1$ to $K_2$ having the property that $H_t|_{V(T)}=\text{id}_{V(T)}$ for all $t \in I$. 
\end{proposition}
\begin{proof} The trefoil knot is fibered.  An explicit fibration is given in \cite{rolfsen}. Let $(T_0,p_0,\Sigma_0)$ denote a fibered triple for this fibration. Let $H:S^3 \times I \to S^3$ denote an ambient isotopy of $S^3$ taking $\Sigma_0$ to the Seifert surface $\Sigma$ depicted in Figure \ref{band_trefoil}. Let $T=H_1(T_0)$ and $p= p_0 \circ H_1^{-1}$. Then $(T,p,\Sigma)$ is a fibered triple.
\newline
\newline
Let $K_1$ be the figure eight knot given by the thin red curve depicted on the far left in Figure \ref{odd_writhe_odds}. Let $K_2$ be the figure eight knot given by the thin red curve depicted on the far left in Figure \ref{odd_writhe_evens}. Knot diagrams $[K_1;T,\Sigma]$ and $[K_2;T,\Sigma]$ on $\Sigma$ are centered in Figures \ref{odd_writhe_odds}, and \ref{odd_writhe_evens} respectively. By Theorem \ref{special_ok}, the virtual covers $\hat{K}_1$ and $\hat{K}_2$ of $K_1$, and $K_2$, respectively, relative to $(T,p,\Sigma)$, are found to be as depicted the far right in Figures \ref{odd_writhe_odds} and \ref{odd_writhe_evens}, respectively. 
\newline
\newline
Computing the odd writhe, we see that $\theta(\hat{K}_1)=2$ and $\theta(\hat{K}_2)=-2$. Thus, $\hat{K}_1 \not\leftrightharpoons \hat{K}_2$ as virtual knots. The result follows from Theorem \ref{main_equiv_thm}.
\end{proof}

\begin{remark} The proof of Propostion \ref{prop_fig_8} also shows that the links $T \sqcup K_1$ and $T \sqcup K_2$ are not unlinked. Indeed, $K_1$ and $K_2$ have virtual covers relative to $(T,p,\Sigma)$ which are non-classical. A virtual cover of each $K_i$ has a non-zero odd writhe.
\end{remark}

\begin{figure}[htb]
\[
\xymatrix{
\begin{array}{c} \scalebox{.5}{\psfig{figure=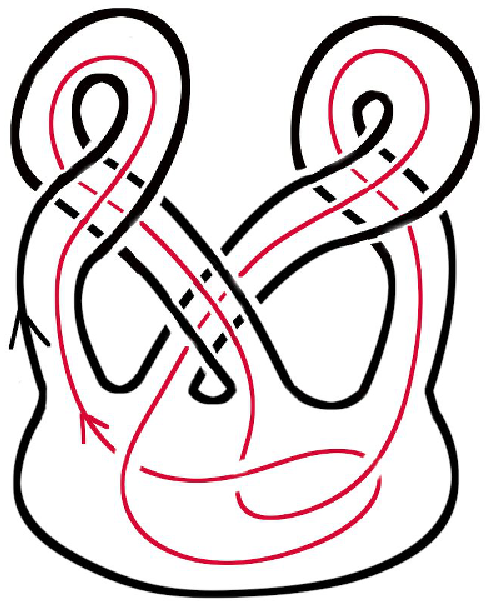}} \end{array} \ar[r] & 
\begin{array}{c}
\scalebox{.5}{\psfig{figure=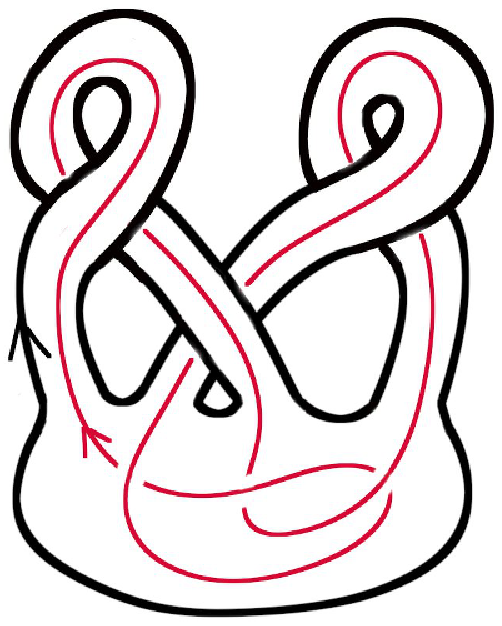}}
\end{array} \ar[r] &
\begin{array}{c} \scalebox{.5}{\psfig{figure=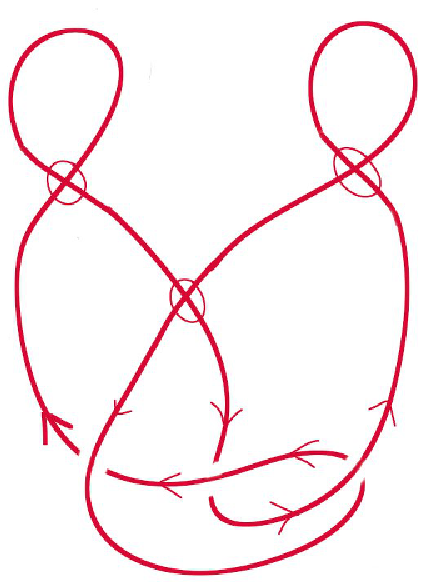}}
\end{array} 
}
\]
\caption{A trefoil $T$ (thick curve) linked with a figure eight knot $K_1$ (thin curve), a special Seifert form $[K_1;T,\Sigma]$, and a virtual cover of $K_1$ relative to $(T,p,\Sigma)$ (far right). The two $\oplus$ signed crossings of the given virtual cover are odd in the Gaussian parity and the $\ominus$ signed crossing is even.} \label{odd_writhe_odds}
\end{figure}

\begin{figure}[htb]
\[
\xymatrix{
\begin{array}{c} \scalebox{.5}{\psfig{figure=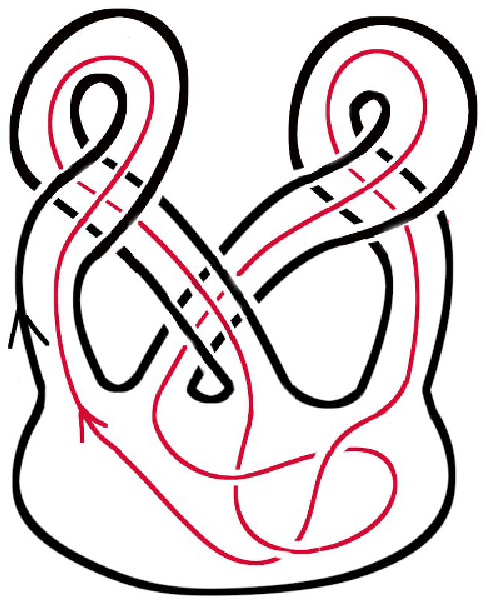}} \end{array} \ar[r] & 
\begin{array}{c}
\scalebox{.5}{\psfig{figure=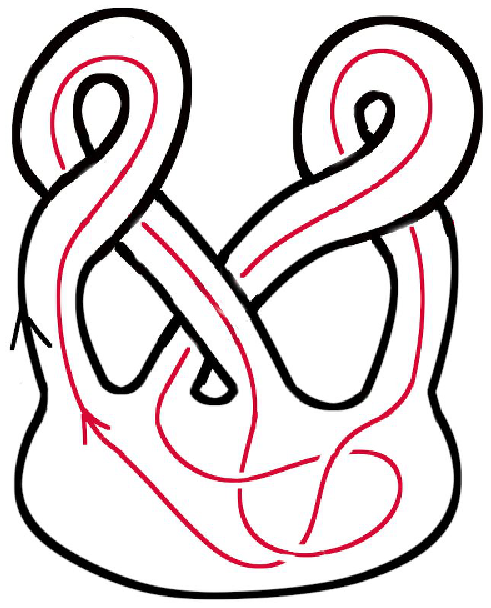}}
\end{array} \ar[r] &
\begin{array}{c} \scalebox{.5}{\psfig{figure=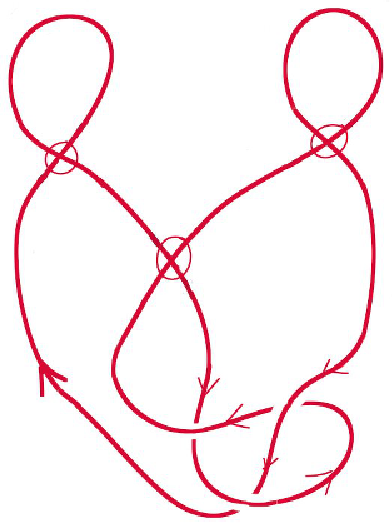}}
\end{array} 
}
\]
\caption{A trefoil $T$ (thick curve) linked with a figure eight knot $K_2$ (thin curve), a special Seifert form $[K_2;T,\Sigma]$, and a virtual cover of $K_2$ relative to $(T,p,\Sigma)$ (far right). The two $\ominus$ signed crossings of the given virtual cover are odd in the Gaussian parity and the $\oplus$ signed crossing is even.}
\label{odd_writhe_evens}
\end{figure}

\subsection{Reproduced Subdiagrams of Classical Knots} \label{sec_irred_odd} A feature of virtual knots is the existence of strong minimality theorems for their diagrams. Recall that a \emph{minimal diagram} of a classical knot is (typically) a diagram having the smallest possible number of classical crossings. A minimal diagram of a classical knot need not be unique: there may be many ``different'' diagrams of the knot which achieve the minimal crossing number. On the other hand, there are virtual knots having diagrams which are minimal in the number of classical crossings and which are reproduced in \emph{all} diagrams of the virtual knot.
\newline
\newline
The aim of the section is to use virtual coverings to demonstrate that there are minimal diagrams for classical knots which are also reproducible in the sense analogous to that of virtual knots. We begin with several definitions.

\begin{definition} [Crossing Change/Virtualization] Let $D$ be a Gauss diagram. A \emph{crossing change} at an arrow $x$ of $D$ is the Gauss diagram $D'$ obtained from $D$ by changing both the direction and the sign of $x$. In an oriented virtual knot diagram, a crossing change at $x$ changes a $\oplus$ classical crossing to an $\ominus$ classical crossing.  A \emph{virtualization} at an arrow $x$ of $D$ is the Gauss diagram $D'$ obtained by changing the direction of $x$ but not the sign of $x$.
\end{definition}

\begin{definition}[Free Knot Diagram] A \emph{free knot diagram} is an equivalence class of Gauss diagrams by crossing changes and virtualizations. A free knot diagram is often depicted as a Gauss diagram with arrowheads and signs erased (see Figure \ref{irred_odd_gauss_fig}, where all the signs on the left hand side are $\oplus$). If $K$ is a virtual knot, the projection of $K$ to a free knot diagram is denoted $[K]$.
\end{definition}

\begin{figure}[htb]
\[
\xymatrix{
\begin{array}{c}\scalebox{.5}{\epsfig{figure=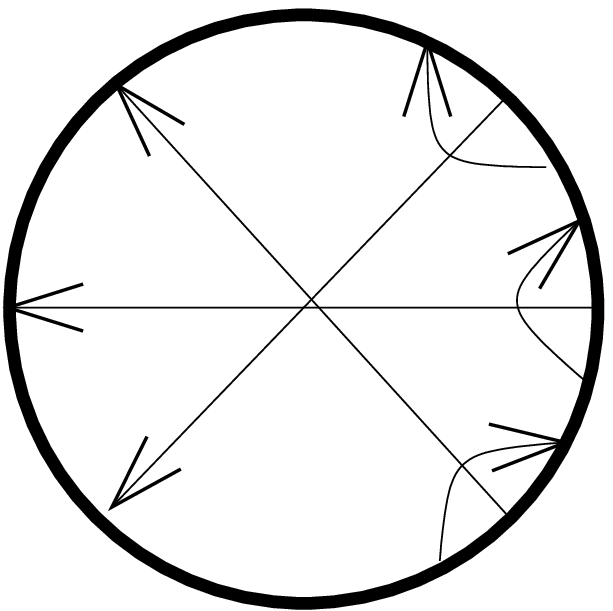}} \end{array} \ar[r] & \begin{array}{c}\scalebox{.5}{\epsfig{figure=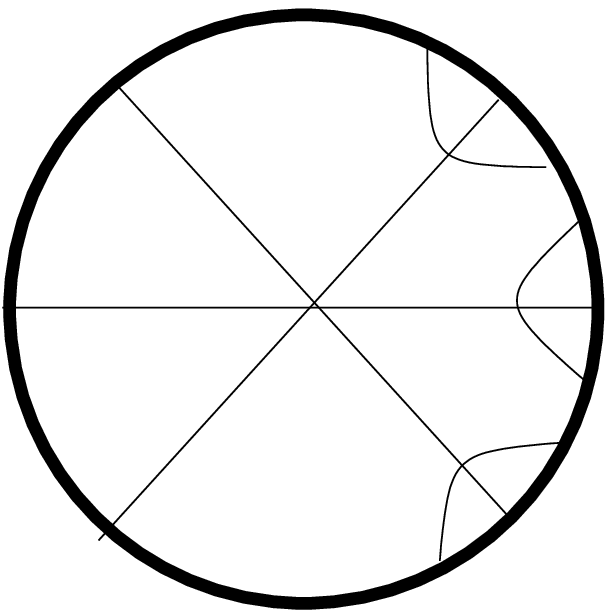}} \end{array} \\
}
\]
\caption{A Gauss diagram (right) and a chord diagram of the projection to an irreducibly odd free knot (left).} \label{irred_odd_gauss_fig}
\end{figure}

\begin{definition}[Free Reidemeister Move] A \emph{free Reidemeister move} is a move $F_l \leftrightharpoons F_r,$ where $R_l \leftrightharpoons R_r$ is an extended Reidemeister move from Figure \ref{exrmoves}, $R_l$ is in the free knot diagram $F_l$, and $R_r$ is in the free knot diagram $F_r$. Free knot diagrams $F_1$ and $F_2$ are said to be \emph{equivalent} if there is a finite sequence of Gauss diagram equivalencies and free Reidemeister moves taking $F_1$ to $F_2$.
\end{definition} 

\noindent A free knot diagram $K$ may be regarded as an immersed graph in $\mathbb{R}^2$.  The vertices of the graph correspond to the crossings of $K$. The edges correspond to arcs of $K$ between the crossings. The \emph{framing} of the free knot diagram is a choice of an Euler circuit of the graph such that consecutive half-edges in the circuit are opposite one another at the crossing where they intersect. Any abstract four valent with a framing is called a \emph{four valent framed graph with one unicursal component} \cite{Ma2}.

\begin{figure}[htb]
\[
\xymatrix{
\begin{array}{c}
\scalebox{.5}{\psfig{figure=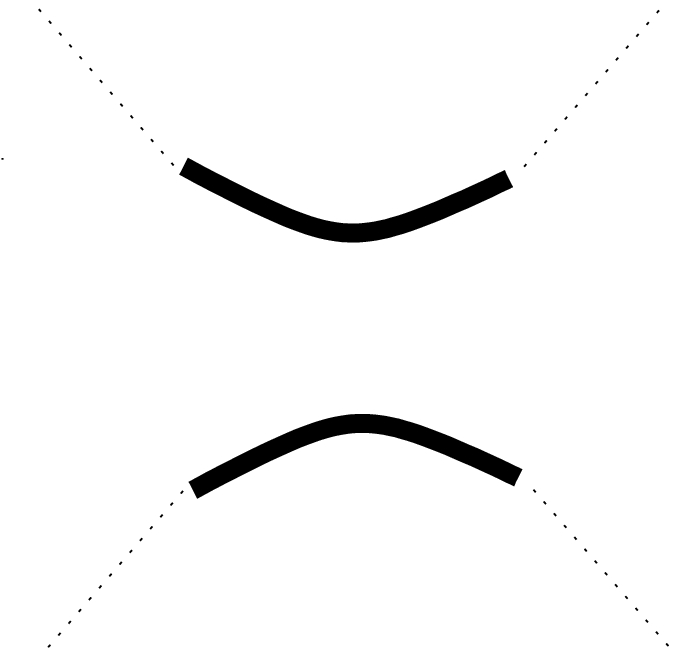}} 
\end{array} &
\begin{array}{c}
\scalebox{.5}{\psfig{figure=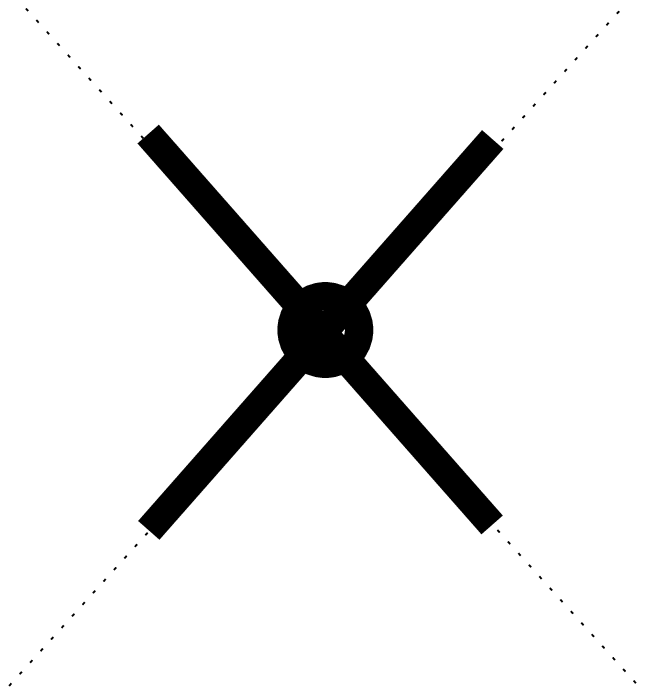}} 
\end{array}
\ar[r]
\ar[l]
&
\begin{array}{c}
\scalebox{.5}{\psfig{figure=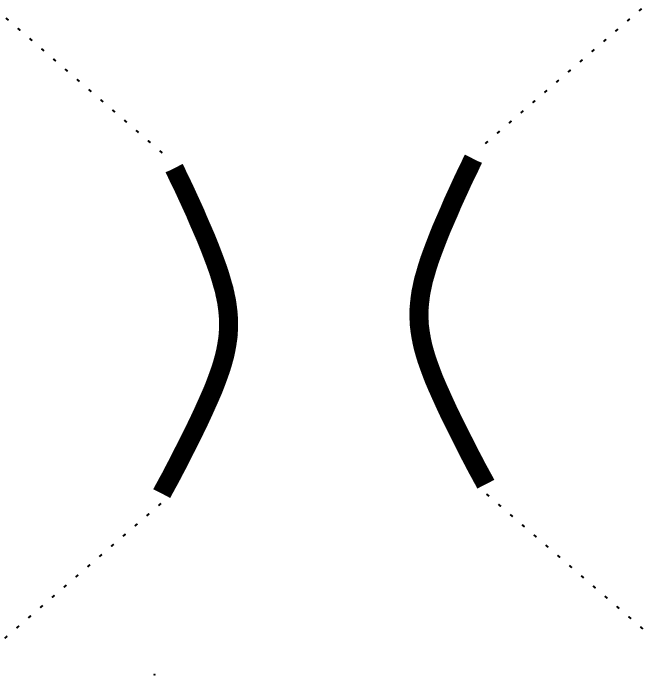}} 
\end{array}
}
\]
\caption{Two types of smoothing for a vertex of a four valent framed graph with one unicursal component.} \label{smooth_fig}
\end{figure}

\begin{definition} Let $G$ be a four valent framed graph with one unicursal component which is immersed in $\mathbb{R}^2$. Let $v$ be a vertex of $G$. By a \emph{smoothing of} $G$ \emph{at} $v$, we mean one of the two modifications of the graph $G$ given in Figure \ref{smooth_fig}. By a \emph{smoothing of} $G$ \emph{at} $S$, we mean the four valent graph obtained by smoothing at each $v \in S$, where $S$ is some subset of the vertices of $G$.
\end{definition}

\begin{definition} A free knot diagram $K$ is said to be \emph{irreducibly odd} if all crossings of $K$ are odd in the Gaussian parity and no decreasing Reidemeister 2 move may be applied to $K$. 
\end{definition}

\noindent The following theorem shows that irreducibly odd diagrams are minimal in the sense that they are ``reproduced'' in all diagrams of the free knot. It was proved by the second named author in \cite{free_knots_and_parity}. The statement below is slightly rephrased from \cite{free_knots_and_parity}.

\begin{theorem}[Manturov \cite{free_knots_and_parity}]\label{free_reproduced} Let $K$ be a four valent framed graph with one unicursal component which is immmersed in $\mathbb{R}^2$. If $K$ represents an irreducibly odd free knot diagram, then for all free knot diagrams $K'$ equivalent to $K$, there is a smoothing of $K'$ which is isomorphic as a graph to $K$.
\end{theorem}

\noindent The above theorem does not provide any interesting information for classical knots, since the universal parity for classical knots is the Gaussian parity \cite{IMN}. However, virtual coverings can be used to show that a non-trivial subdiagram of a knot $K$ is ``reproduced'' in every diagram of $K$ which is ``close'' to $K$. A subdiagram will be ``reproduced'' in exactly the same sense as in Theorem \ref{free_reproduced}. The imprecise notions of ``reproduced'' and ``close'' are made precise by the following definition.

\begin{definition}[$(J,\Sigma)$-bound of $K$] Let $J$ be a knot and $\Sigma$ a Seifert surface for $J$. Let $K$ be a knot in special Seifert form $(K;J,\Sigma)$. The $(J,\Sigma)$-\emph{bound} of $K$ is the set of knots $K'$ in special Seifert form $(K';J,\Sigma)$ such that $K' \leftrightharpoons K$ by an ambient isotopy $H:S^3 \times I \to S^3$ such that $H_t|_{V(J)}=\text{id}_{V(J)}$ for all $t \in I$.
\end{definition}

\begin{remark} Suppose $K'$ is in the $(J,\Sigma)$-bound of $K$ and $H$ is an ambient isotopy taking $K$ to $K'$ such that $H_t|_{V(J)}=\text{id}_{V(J)}$ for all $t \in I$. It is certainly true that $H_1(K)=K'$. However, It is quite possible that $H_1(\Sigma) \ne \Sigma$. We note that to use Theorem \ref{main_equiv_thm}, we only need that $K$ has the special Seifert form $(K;J,\Sigma)$ and $K'$ has the special Seifert form $(K';J,\Sigma)$. It is not necessary that $H_1(\Sigma)=\Sigma$ for the conclusion of Theorem \ref{main_equiv_thm} to hold.
\end{remark}

\begin{theorem}\label{classical_reproduced} Let $\hat{K}$ be a virtual cover of a knot $K$ relative to $(J,p,\Sigma)$, where $(K;J,\Sigma)$ is in special Seifert form in $\overline{S^3\backslash V(J)}$. Suppose that $K_0$ is in the $(J,\Sigma)$-bound of $K$ and that $\hat{K}_0$ is a virtual cover of $K_0$ relative to $(J,p,\Sigma)$.
\begin{enumerate}
\item If $[\hat{K}]$ is irreducibly odd, then there is a smoothing of $[\hat{K}_0]$ which is isomorphic as a graph to $[\hat{K}]$. 
\item The number of crossings of the diagram $[K;J,\Sigma]$ on $\Sigma$ is less than or equal to the number of crossings of the diagram $[K_0;J,\Sigma]$ on $\Sigma$.
\end{enumerate}
\end{theorem}
\begin{proof} This follows immediately from the definitions, Theorem \ref{main_equiv_thm}, and Theorem \ref{free_reproduced}.
\end{proof}

\begin{figure}[h]
\[
\xymatrix{
\begin{array}{c} \scalebox{.7}{\psfig{figure=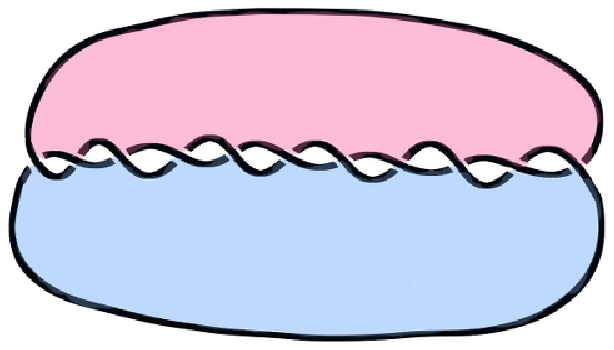}} \end{array} \ar[r] & \begin{array}{c} \scalebox{.7}{\psfig{figure=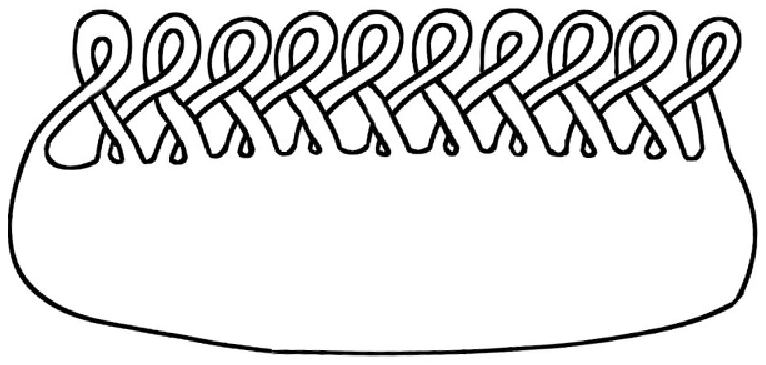}} \end{array}
}
\]
\caption{A fiber $\Sigma_0$ (left) of the fibered triple $(J_0,p_0,\Sigma_0)$ and an ambient isotopy $H:S^3 \times I \to S^3$ taking $\Sigma_0$ to $\Sigma_3$ (right). $\Sigma_0$ is a Seifert surface for $11a_{367}$ \cite{KnotInfo,KnotAtlas}.} \label{band_eleven_a_349}
\end{figure}

\begin{figure}[htb]
\scalebox{.75}{\epsfig{figure=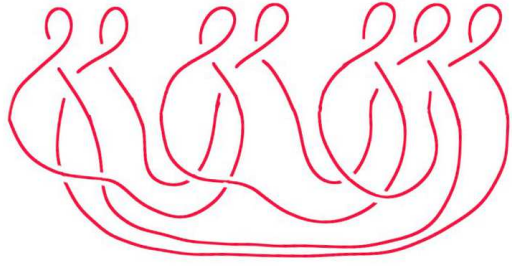}}
\caption{A diagram of $K_3$ as it appears in the given embedding in $S^3$. In $S^3$, $K_3$ is unknotted.} \label{irred_odd_unknot_fig}
\end{figure}

\begin{proposition} There exists a diagram $K_3$ of the unknot, and a special Seifert form $(K_3;J_3,\Sigma_3)$, where $J_3 \leftrightharpoons 11a_{367}$, such that if $K_3'$ is in the $(J_3,\Sigma_3)$-bound of $K_3$, there is a smoothing of $[\kappa([K_3';J_3,\Sigma_3])]$ which is isomorphic as a graph to $[\kappa([K_3;J_3,\Sigma_3])]$. Moreover, the diagram $[K_3;J_3,\Sigma_3]$ on $\Sigma_3$ has the smallest number of crossings of all diagrams $[K_3';J_3,\Sigma_3]$ on $\Sigma_3$, where $(K_3';J_3,\Sigma_3)$ is in the $(J_3,\Sigma_3)$-bound of $K_3$.
\end{proposition}
\begin{proof} The knot $11a_{367}$ is fibered. A particular fibration can be found by an obvious generalization of the fibration of the trefoil given by Rolfsen \cite{rolfsen}. Let $(J_0,p_0,\Sigma_0)$ denote a fibered triple for this fibration.  There is an ambient isotopy $H:S^3 \times I \to S^3$ taking $\Sigma_0$ to the surface $\Sigma_3$ depicted on the right hand side of Figure \ref{band_eleven_a_349}. We define a fibered triple via this ambient isotopy: $(J_3,p_3,\Sigma_3)=(H_1(J_0),p_0 \circ H_1^{-1},\Sigma_3)$.
\newline
\newline
Let $K_3$ denote the diagram of the unknot depicted in Figure \ref{irred_odd_unknot_fig}. In Figure \ref{irred_odd_surf_fig}, we have a special Seifert form $[K_3;J_3,\Sigma_3]$. By Theorem \ref{special_ok}, a virtual cover $\hat{K}_3$ of $K_3$ is given in Figure \ref{irred_odd_gamma_fig}. A Gauss diagram of $\hat{K}_3$ is given on the left hand side of Figure \ref{irred_odd_gauss_fig}, where all arrows are signed $\oplus$. Then $[\hat{K}_3]$ is irreducibly odd. The result follows from Theorems \ref{classical_reproduced} and \ref{main_equiv_thm}.
\end{proof}

\begin{figure}[htb]
\scalebox{.6}{\epsfig{figure=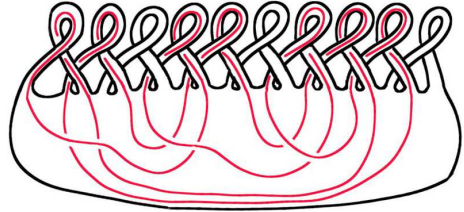}}
\caption{A special band form $[K_3;J_3,\Sigma_3]$ where $J_3$ is the fibered knot $11a_{367}$ \cite{KnotAtlas} and $K_3$ is a trivial knot in $S^3$} \label{irred_odd_surf_fig} 
\end{figure}

\begin{figure}[htb]
\scalebox{.6}{\epsfig{figure=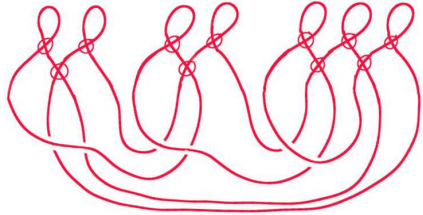}}
\caption{A diagram of $\hat{K}_3$. The free knot $[\hat{K}_3]$ is irreducibly odd.} \label{irred_odd_gamma_fig}
\end{figure}

\subsection{Isotopies of Invertible Knots} \label{sec_non_invert} Virtual covers may also be used to investigate ambient isotopies of invertible knots.  Recall that the \emph{inverse} of an oriented knot $K$ is the oriented knot obtained from $K$ by changing the orientation of the the knot. If $K$ is an oriented knot, its inverse is denoted $K^{-1}$. An oriented knot is said to be \emph{invertible} if it is ambient isotopic to its inverse. An oriented knot is said to be non-invertible if it is not invertible. Non-invertible knots were first discovered by Trotter \cite{trotter_noninvert}. Non-invertible links with invertible components were first discovered by Whitten \cite{whitten_sublinks}.

\begin{figure}[htb]
\[
\xymatrix{
\begin{array}{c} \scalebox{.6}{\psfig{figure=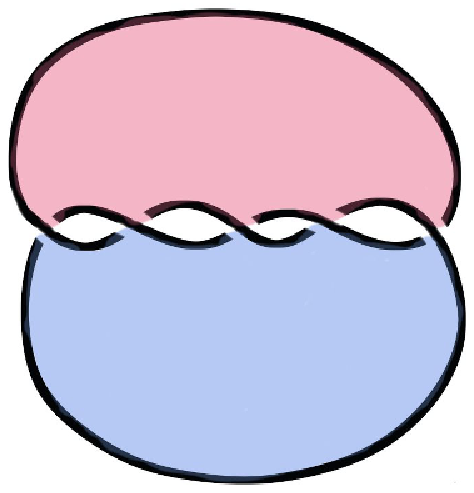}} \end{array} \ar[r]^{H} & \begin{array}{c} \scalebox{1}{\psfig{figure=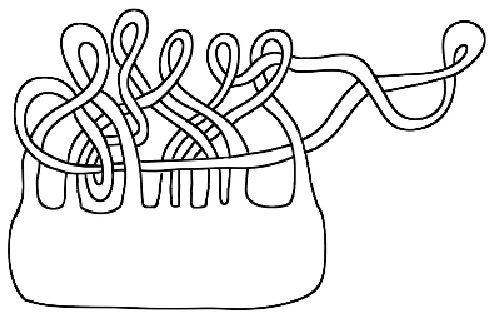}} \end{array}
}
\]
\caption{A fiber $\Sigma_0$ (left) of the fibered triple $(J_0,p_0,\Sigma_0)$ and an ambient isotopy $H:S^3 \times I \to S^3$ taking $\Sigma_0$ to $\Sigma_4$ (right).} \label{band_five_one_improper}
\end{figure}

\begin{figure}[htb]
\[
\xymatrix{
\begin{array}{c}\scalebox{.7}{\epsfig{figure=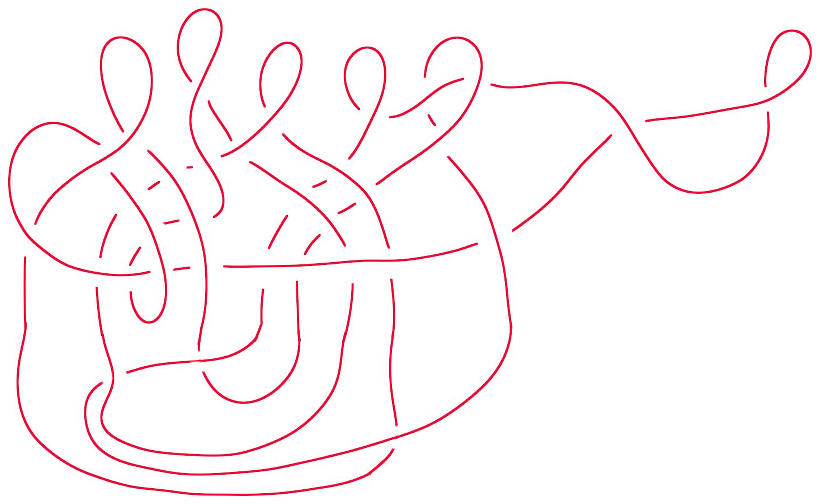}} \end{array} \ar[r] & \ar[l] \begin{array}{c}\scalebox{.5}{\epsfig{figure=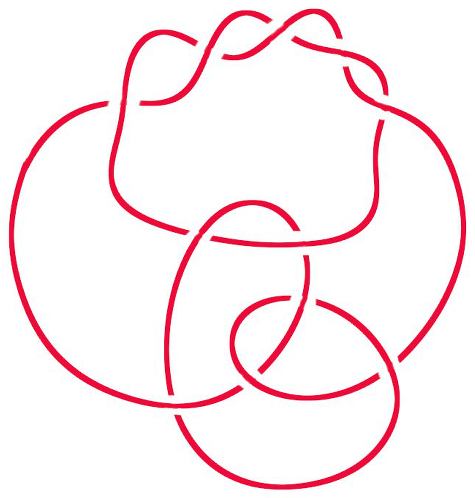}} \end{array} \\
}
\]
\caption{The classical knot $K_4$ (left) and a more recognizable diagram of $K_4$, drawn as the mirror image of $11\text{a}_9$ \cite{KnotInfo}(right).} \label{not_invert_class_fig}
\end{figure}

\begin{proposition}\label{invert_example} There is an invertible knot $K_4$ in $S^3$, and a fibered knot $J_4$ in the complement of $K_4$ such that $\text{lk}(J_4,K_4)=0$ and there is no ambient isotopy $H:S^3 \times I \to S^3$ taking $K_4$ to $K_4^{-1}$ such that $H_t|_{V(J_4)}=\text{id}_{V(J_4)}$ for all $t \in I$. The existence is satisfied by a knot $K_4$ which is equivalent to the mirror image of $11a_9$ and a fibered knot $J_4$ which is equivalent to $5_1$.  
\end{proposition}
\begin{proof} It is well known that $5_1$ is a fibered knot. A particular fibration of $5_1$ can be found using a natural generalization to the case of the trefoil given by Rolfsen \cite{rolfsen}. Let $(J_0,p_0,\Sigma_0)$ denote this particular fibration, where $J_0 \leftrightharpoons 5_1$. There is an ambient isotopy $H:S^3 \times I \to S^3$ taking $\Sigma_0$ to the surface $\Sigma_4$ depicted on the right hand side of Figure \ref{band_five_one_improper}. Then $(H_1(J_0), p \circ H_1^{-1},\Sigma_4)=(J_4,p_4,\Sigma_4)$ is a fibered triple.
\newline
\newline
Let $K_4$ denote the knot on the left hand side of Figure \ref{not_invert_class_fig}. There is a sequence of Reidemeister moves taking $K_4$ to the mirror image of $11a_9$ (see right hand side of Figure \ref{not_invert_class_fig}). Figure \ref{not_invert_surf_fig} shows $K_4$ as a diagram $[K_4;J_4,\Sigma_4]$ on $\Sigma_4$. Using Theorem \ref{special_ok}, we can find the virtual covers of $K_4$ relative to $(J_4,p_4,\Sigma_4)$. The virtual cover $\hat{K}_4$ is given on the left hand side of Figure \ref{not_invert_virt_fig}. On the right hand side, a simpler diagram appears. 
\newline
\newline
We give $\hat{K}_4$ an orientation. In particular, we orient the over-crossing arc of the leftmost classical crossing in $\hat{K}_4$ from left to right. With these conventions, we compute the normalized Sawollek polynomial \cite{saw}:

\begin{eqnarray*}
\tilde{Z}_{\hat{K}_4}(x,y) &=& x^2-x^3+\frac{x^2}{y}-\frac{x^3}{y}+xy-y^2+xy^2, \\
\tilde{Z}_{\hat{K}_4^{-1}}(x,y) &=& 1-x-\frac{x^2}{y}+\frac{x^3}{y^2}-\frac{x^2}{y}+\frac{x^3}{y}+y-xy.
\end{eqnarray*}
\hspace{1cm}
\newline
The stated claim now follows from Theorem \ref{main_equiv_thm}. The simplified diagram of the knot $11a_9$ was verified by consulting KnotInfo \cite{KnotInfo} and KnotAtlas \cite{KnotAtlas}. 
\end{proof}

\begin{figure}[htb]
\scalebox{.6}{\epsfig{figure=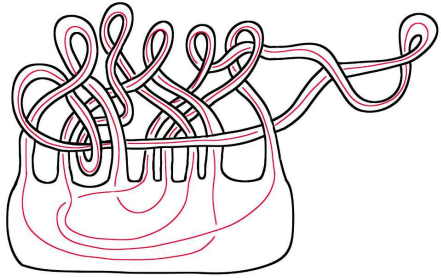}}
\caption{A special Seifert form $[K_4;J_4,\Sigma_4]$ where $K_4$ is the mirror image of $11\text{a}_9$ and $J_4$ is equivalent to $5_1$.} \label{not_invert_surf_fig} 
\end{figure}

\begin{figure}[htb]
\[
\xymatrix{
\begin{array}{c}\scalebox{.5}{\epsfig{figure=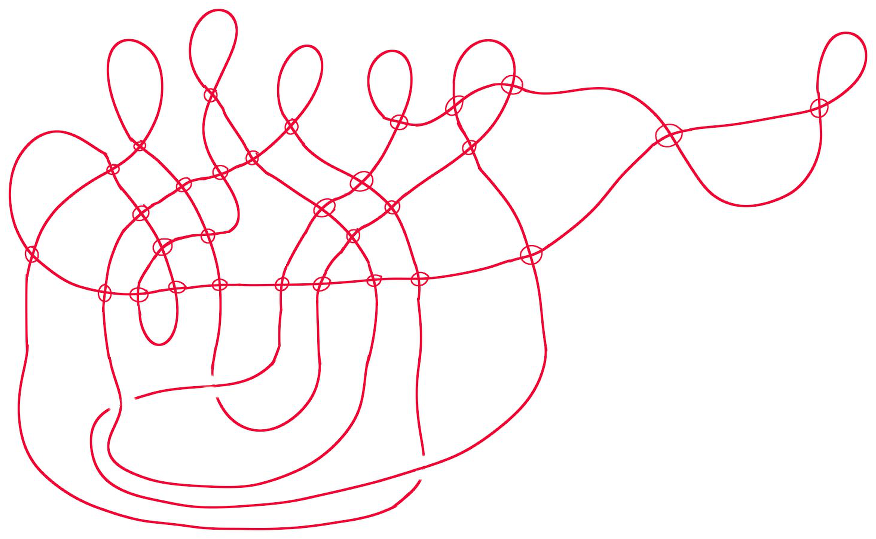}} \end{array} \ar[r] & \ar[l] \begin{array}{c}\scalebox{.5}{\epsfig{figure=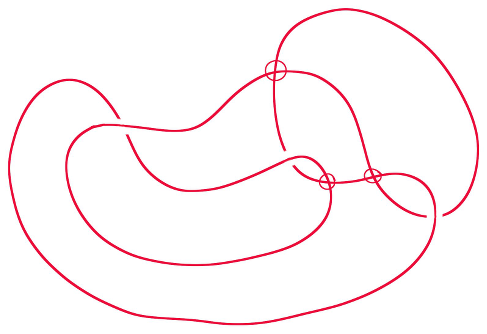}} \end{array} \\
}
\]
\caption{The virtual cover $\hat{K}_4$ of $K_4$ relative to $(J_4,p_4,\Sigma_4)$ and an equivalent virtual knot with fewer virtual crossings.} \label{not_invert_virt_fig}
\end{figure}

\section{Virtual Knots and Virtual Coverings}
\subsection{Every Virtual Knot is a Virtual Cover of Some Knot}\label{sec_every_virt} We prove that for every virtual knot $W$, there is a classical knot $K$ and a fibered triple $(J,p,\Sigma)$ such that $\text{lk}(J,K)=0$ and $W$ is a virtual cover of $K$ relative to $(J,p,\Sigma)$. The idea of the proof is to represent $W$ as a knot diagram on a surface $\Sigma$ of sufficiently high genus which is a Seifert surface of an unknot. Then we make a sequence of ``moves'' on the knot diagram and surface (simultaneously) so that the linking number and the underlying virtual knot $W$ do not change. A sequence of moves is made until we have a Seifert surface which coincides with a fiber of some fibered knot. We begin with the following definition from the literature.  

\begin{definition} A \emph{disk-band presentation} \cite{bz} of a Seifert surface $\Sigma$ of a knot $J$ in $\mathbb{R}^3$ is a decomposition $\Sigma=D^2 \cup R_1 \cup \cdots \cup R_{2h}$ of $\Sigma$ into a disk $D^2$ and $2h$ pairwise disjoint rectangles $R_i \approx I \times I$ such that $R_i \cap D^2$ consists of two disjoint arcs $a_i$, $a_i'$ in $\partial D^2$ corresponding to opposite sides of the rectangle $R_i$. Moreover, it is required that the arcs alternate around $D^2$. This means that there is a consistent labeling and a choice of basepoint on $\partial D^2$ so that the arcs appear as $a_1,a_2,a_1',a_2',a_3,\ldots,a_{2h-1},a_{2h},a_{2h-1}',a_{2h}'$ when traveling along $\partial D^2$ from the base point. Finally, for every disk-band presentation, there is a projection $r:\mathbb{R}^3 \to \mathbb{R}^2$ so that $r|_{\Sigma}$ is a local homeomorphism. The condition that $r|_{\Sigma}$ is a local homeomorphism guarantees that none of the bands of the decomposition contains any twists. It is well-known that any Seifert surface $J$ has a disk-band presentation \cite{bz}.
\end{definition}

\noindent Suppose that $(K;J,\Sigma)$ is a special Seifert form, and $\Sigma$ is given as a disk-band presentation. In this situation, we define two types of moves on the diagram $[K;J,\Sigma]$ on $\Sigma$: the loop move and the pass move.  Note that the moves are well-defined even if $J$ is not fibered. 

\begin{definition}[The Loop Move] The \emph{loop move} is the modification to $\Sigma$ and $[K;J,\Sigma]$ depicted in Figure \ref{loop_move}. The figure represents a small portion of a band of $\Sigma$ and the arcs of the knot diagram $[K;J,\Sigma]$. Outside the indicated neighborhood of the band portion, no modification is made to the surface or knot diagram. Note that the loop move does not change the homeomorphism type of $\Sigma$ but the move may make a nontrivial modification $J \to J'$ of the boundary knot so that $J \not\leftrightharpoons J'$.
\end{definition}

\begin{figure}[htb]
\[
\begin{array}{c}
\scalebox{.5}{\psfig{figure=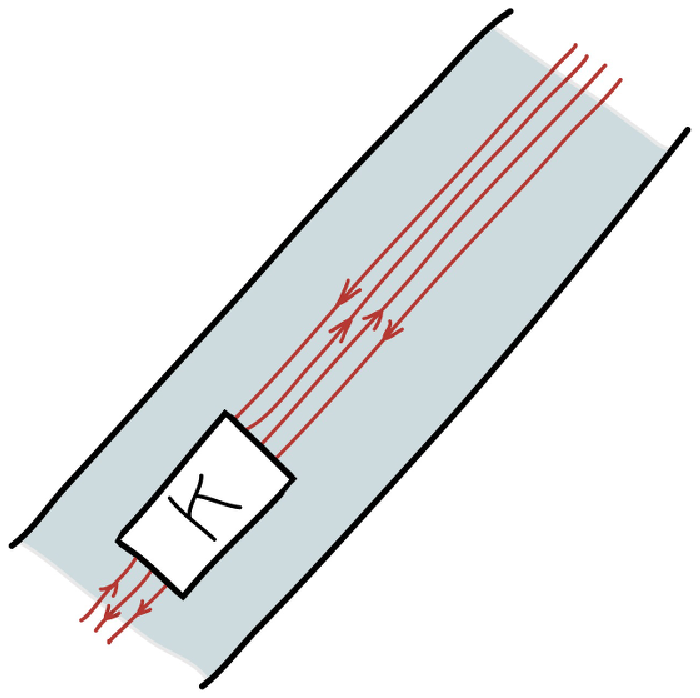}}
\end{array}
\leftrightharpoons
\begin{array}{c}
\scalebox{.5}{\psfig{figure=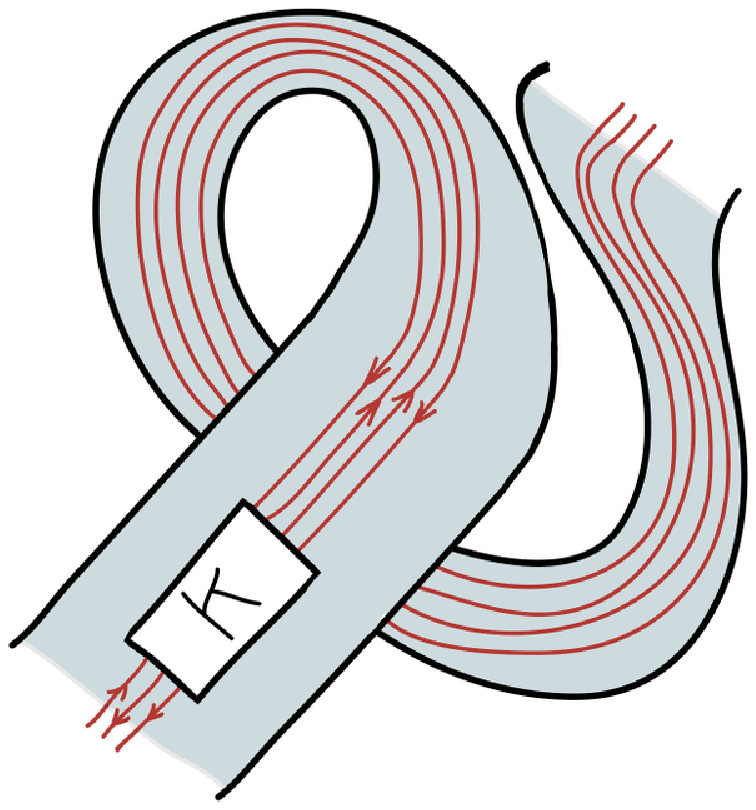}}
\end{array}
\]
\caption{The loop move.}\label{loop_move}
\end{figure}

\begin{definition}[The Pass Move] The \emph{pass move} is the modification to $\Sigma$ and $[K;J,\Sigma]$ depicted in Figure \ref{pass_move}. The figure represents a small portions of bands of $\Sigma$ and the arcs of the knot diagram $[K;J,\Sigma]$. No modification to the surface or knot diagram is made outside of the indicated neighborhood. The portions of bands may represent different portions of the same band of $\Sigma$. Note that the pass move does not change the homeomorphism type of $\Sigma$, but can make a nontrivial modification to the knot type of its boundary.
\end{definition} 

\begin{figure}[htb]
\[
\begin{array}{c}
\scalebox{.75}{\psfig{figure=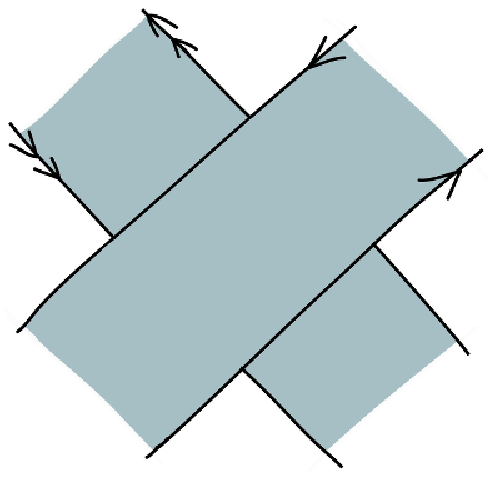}}
\end{array}
\leftrightharpoons
\begin{array}{c}
\scalebox{.75}{\psfig{figure=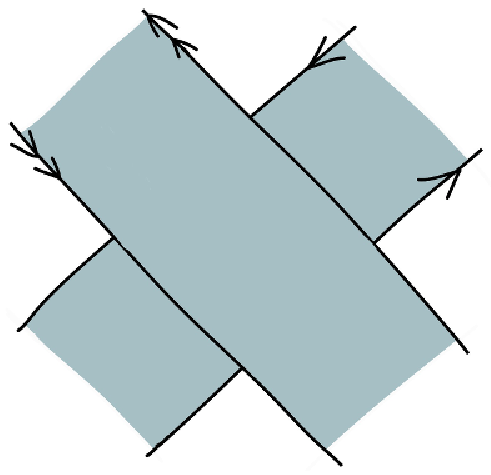}}
\end{array}
\]
\caption{The pass move on bands.}\label{pass_move}
\end{figure}

\begin{lemma} \label{pass_loop_preserve} If $[K_1;J_1,\Sigma_1]$ is obtained from $[K_0;J_0,\Sigma_0]$ by a loop move or a pass move, then we have:
\begin{eqnarray*}
\text{lk}(J_1,K_1) &=& \text{lk}(J_0,K_0), \text{ and} \\
\kappa([K_1;J_1,\Sigma_1]) &=& \kappa([K_0;J_0,\Sigma_0]).
\end{eqnarray*}
\end{lemma}
\begin{proof} In either a loop move or a pass move, the knot diagrams $[K_1;J_1,\Sigma_1]$ and $[K_0;J_0,\Sigma_0]$ have an identical Gauss diagram. This proves that the second equation holds.
\newline
\newline
The pass move certainly preserves the linking number. The first equation must be established for the loop move. Consider each of the arcs of $K_0$ depicted in the loop move. In the projection of $J_0 \sqcup K_0$, each such arc has four more crossings on the left side of the move than on the right side of the move. Now, if $J_0$ is oriented, then the sides on each band of $\Sigma$ are oppositely oriented. Hence, the four additional crossings contribute two $\oplus$ crossings and two $\ominus$ crossings. Therefore the linking number is not changed by the loop move.
\end{proof}

\begin{remark} \label{remark_lk_0_2} The previous lemma can be used to give a combinatorial proof that if $(K;J,\Sigma)$ is a special Seifert form, then $\text{lk}(J,K)=0$. Indeed, we assume that $\Sigma$ is given in disk-band presentation. We have a projection $r:\mathbb{R}^3 \to \mathbb{R}^2$ such that $r|_{\Sigma}$ is a local homeomorphism. The linking number can be computed from this projection. Note that the only contributions to the linking number occur where bands of $\Sigma$ cross one another (or themselves). It is easy to see that every $\oplus$ contribution to the linking number must then have a corresponding $\ominus$ contribution to the linking number. Whence, $\text{lk}(J,K)=0$.  
\end{remark}

\begin{figure}[htb]
\scalebox{.5}{\epsfig{figure=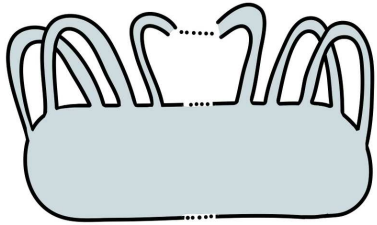}}
\caption{The surface $S_h$ used in the proof of Theorem \ref{onto}.} \label{s_h}
\end{figure}

\begin{theorem} \label{onto} For every virtual knot $W$, there is a classical knot $K$ and a fibered triple $(J,p,\Sigma)$ such that $\text{lk}(J,K)=0$ and $W$ is a virtual cover of $K$ relative to $(J,p,\Sigma)$.
\end{theorem}
\begin{proof} If $W$ is classical, then the theorem follows from Theorem \ref{unlinked_implies_classical}. Suppose that $W$ is non-classical. Let $S_h$ denote the disk-band presentation of the Seifert surface of the unknot given in Figure \ref{s_h}, $S_h=D^2 \cup R_1 \cup \cdots \cup R_{2h}$. There exists an $h_0>0$ such that $W$ is represented as a knot diagram $W'$ on $S_{h_0}$.
\newline
\newline
Let $J_0$ be a fibered knot of genus $h_0$ with fibered triple $(J_0,p_0,\Sigma_0)$. We take an ambient isotopy of $H:S^3 \times I \to S^3$ so that $\Sigma_0$ is represented in band presentation. Let $J=H_1(J_0)$, $\Sigma=H_1(\Sigma_0)$, and $p=p_0 \circ H_1^{-1}$. 
\newline
\newline
Now, since both $S_{h_0}$ and $\Sigma$ are in band presentation, it follows that there is a sequence of ambient isotopies of $S^3$, loop moves, and pass moves taking $W'$ on $S_{h_0}$ to a diagram $W''$ on $\Sigma$ (see Figure \ref{onto_example_2}). The diagram $W''$ on $\Sigma$ corresponds to a knot diagram $[K;J,\Sigma]$ on $\Sigma$ of a special Seifert form $(K;J,\Sigma)$. Thus, by Lemma \ref{pass_loop_preserve}, Remarks \ref{remark_lk_0_1} and \ref{remark_lk_0_2}, and Theorem \ref{special_ok}, we have that $W$ is a virtual cover for $K$ relative to $(J,p,\Sigma)$.
\end{proof}

\begin{figure}[htb]
\[
\begin{array}{|c|c|} \hline
\underline{\text{The Kishino Knot } W:} & \underline{W \text{ as a diagram } W' \text{ on } S_2:} \\
\begin{array}{c} 
\scalebox{.65}{\psfig{figure=kishino_knot.eps}} \end{array} & \begin{array}{c} \scalebox{.65}{\psfig{figure=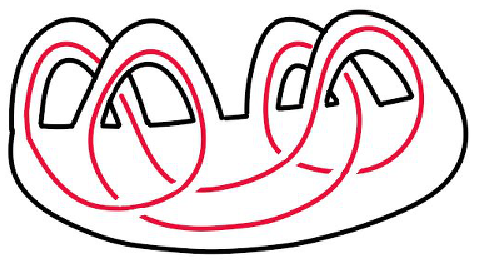}} \end{array} \\ \hline
\multicolumn{2}{|c|}{\underline{\text{A special Seifert form } [K;J,\Sigma], J \leftrightharpoons 5_1, \Sigma \text{ a fiber}:}}\\
\multicolumn{2}{|c|}{
\begin{array}{c} \scalebox{1}{\psfig{figure=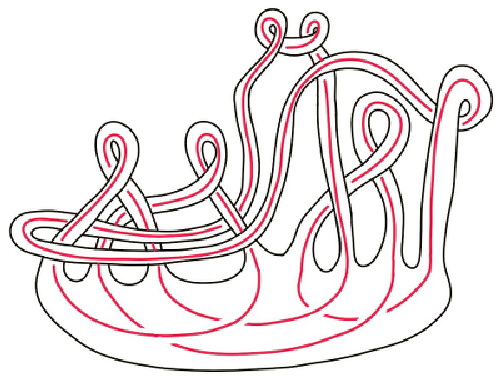}}\end{array}} \\ \hline
\end{array}
\]
\caption{Three steps in finding a classical knot having the Kishino knot as a virtual cover. Observe that $W=\kappa([K;J,\Sigma])$. } \label{onto_example_2}
\end{figure}

\subsection{Acknowledgement} The authors are indebted to the the anonymous reviewer for pointing out that the results of the present paper are better suited for the smooth category than the p.l. category. In particular, the reviewer noted that the hypothesis of local unknottedness (see Remark \ref{loc_unknot}) is not needed to prove Lemma \ref{isotopies_lift} and Theorem \ref{main_equiv_thm}. In the smooth category, these results follow from the isotopy extension theorem (see \cite{hirsch}, Chapter 8, Theorem 1.3). This observation greatly improved both the exposition and the quality of the results.

\bibliographystyle{plain}
\bibliography{bib_fiber}

\end{document}